\documentclass[leqno]{article}

\usepackage{CJKutf8}
\usepackage{a4wide}
\usepackage[]{amsfonts} \usepackage[]{amsmath}
\usepackage[]{amssymb} \usepackage[]{latexsym}
\usepackage{graphicx,color} \usepackage{amsthm}
\usepackage{url}
\usepackage{graphicx}
\usepackage[numbers]{natbib}
\usepackage{bm}
\usepackage{amssymb,amsmath,color,times}
 \usepackage{mathptmx}      
\usepackage{latexsym}
  \usepackage[title]{appendix}

\begin{document}
\begin{CJK}{UTF8}{gbsn}

\theoremstyle{plain}
\newtheorem{theorem}{Theorem}[section] \newtheorem*{theorem*}{Theorem}
\newtheorem{proposition}[theorem]{Proposition} \newtheorem*{proposition*}{Proposition}
\newtheorem{lemma}[theorem]{Lemma} \newtheorem*{lemma*}{Lemma}
\newtheorem{corollary}[theorem]{Corollary} \newtheorem*{corollary*}{Corollary}

\theoremstyle{definition}
\newtheorem{definition}[theorem]{Definition} \newtheorem*{definition*}{Definition}
\newtheorem{example}[theorem]{Example} \newtheorem*{example*}{Example}
\newtheorem{remark}[theorem]{Remark} \newtheorem*{remark*}{Remark}
\newtheorem{problem}[theorem]{Problem} \newtheorem*{problem*}{Problem}
\newtheorem{hypotheses}[theorem]{Hypotheses} \newtheorem{assumption}[theorem]{Assumption}
\newtheorem{notation}[theorem]{Notation} \newtheorem*{question}{Question}

\newcommand{\ds}{\displaystyle} \newcommand{\nl}{\newline}
\newcommand{\eps}{\varepsilon}
\newcommand{\bE}{\mathbb{E}}
\newcommand{\cB}{\mathcal{B}}
\newcommand{\cF}{\mathcal{F}}
\newcommand{\cA}{\mathcal{A}}
\newcommand{\cM}{\mathcal{M}}
\newcommand{\cD}{\mathcal{D}}
\newcommand{\cH}{\mathcal{H}}
\newcommand{\cN}{\mathcal{N}}
\newcommand{\cL}{\mathcal{L}}
\newcommand{\cLN}{\mathcal{LN}}
\newcommand{\bP}{\mathbb{P}}
\newcommand{\bQ}{\mathbb{Q}}
\newcommand{\bN}{\mathbb{N}}
\newcommand{\bR}{\mathbb{R}}
\newcommand{\barsigma}{\overline{\sigma}}
\newcommand{\VIX}{\mbox{VIX}}
\newcommand{\erf}{\mbox{erf}}
\newcommand{\LMMR}{\mbox{LMMR}}
\newcommand{\cLcir}{\mathcal{L}_{\!_{C\!I\!R}}}
\newcommand{\rhor}{\raisebox{1.5pt}{$\rho$}}
\newcommand{\varphir}{\raisebox{1.5pt}{$\varphi$}}
\newcommand{\taur}{\raisebox{1pt}{$\tau$}}
\newcommand{\spx}{S\&P 500 }

\title{Malliavin calculus and its application to robust optimal portfolio for an insider}

\author{Chao Yu \and Yuhan Cheng}

\maketitle

\begin{abstract}
Insider information and model uncertainty are two unavoidable problems for the portfolio selection theory in reality. This paper studies the robust optimal portfolio strategy for an investor who owns general insider information under model uncertainty. On the aspect of the mathematical theory, we improve some properties of the forward integral and use Malliavin calculus to derive the anticipating It\^{o} formula  . Then we use forward integrals to formulate the insider-trading problem with model uncertainty. We give the half characterization of the robust optimal portfolio and obtain the semimartingale decomposition of the driving noise $W$ with respect to the insider information filtration, which turns the problem turns to the nonanticipative stochastic differential game problem. We give the total characterization by the stochastic maximum principle. When considering two typical situations where the insider is `small' and `large', we give the corresponding BSDEs to characterize the robust optimal portfolio strategy, and derive the closed form of the portfolio and the value function in the case of the small insider by the Donsker $\delta$ functional. We present the simulation result and give the economic analysis of optimal strategies under different situations.
\end{abstract}

\section{Introduction}

\label{intro}
The optimal portfolio strategy problem is an important subject in mathematical finance, which was first introduced by \cite{Merton69,Merton71}. The main purpose is to select a portfolio strategy $\pi$, i.e., the proportion of the investor's total wealth $X^\pi$ invested in the risky asset $S$ (given that there is a risk-free asset and a risky asset in the financial market), to maximize her terminal expected utility as follows
\begin{equation}\label{intro_equ1}
\begin{aligned}
 \max_{\pi}\mathbb{E}\left[ U(X^\pi_T)  \right].
\end{aligned}\end{equation}
where $U(x)$ is the utility function, and $T>0$ is some fixed terminal time. In continuous time market, the martingale method, the dynamic programming method (with Hamilton-Jacobi-Bellman (HJB) equations) and the stochastic maximum principle (with backwards stochastic differential equations (BSDEs)) are three traditional methods, based on the classical It\^{o} theory (see \cite{Karatzas91}), to solve this problem (see \cite{Merton69,Merton71,Yong99}).  Many extended models and problems have been studied in recent decades (see \cite{Karatzas98,Oksendal19,Gu20}).

Recently, there is a growing emphasis on the optimization problem of insider trading. That is, the investor who owns additional insider information. In this setting, we naturally assume that the portfolio process $\pi$ is adapted to the insider information filtration $\{{\cal H}_t\}_{0\le t\le T}$ , which is larger than the original filtration $\{{\cal F}_t\}_{0\le t\le T}$ generated by the Brownian motion $W$ (the driving noise of the risky asset $S$). Thus the relevant It\^{o} stochastic differential equations (SDEs) in the model should be replaced by the anticipating stochastic differential equations, which implies that the above traditional methods could not be applied directly. 

\cite{Karatzas96} was the first to study the optimization problem of insider trading. They assumed the insider information is hidden in a random variable $Y$ from the beginning of filtrations. Thus the insider information filtration is of the initial enlargement type, i.e., 
\begin{equation}\label{initialtype}
{\cal H}_t=\bigcap_{s>t}\left({\cal F}_s\vee Y\right).
\end{equation}
They used the technique of enlargement of filtration (see \cite{DiNunno09}), which implies that $W$ is a semimartingale with respect to the enlarged filtration $\{{\cal H}_t\}_{0\le t\le T}$ under some specific conditions. However, when these conditions are removed or the filtration is of other forms, $W$ is not necessarily a semimartingale with respect to $\{{\cal H}_t\}_{0\le t\le T}$. \cite{Biagini05} first developed a method based on the theory of forward integrals, which is the natural generalization of the It\^{o} integral, and considered the most general insider information filtration without the technique of enlargement of filtration. \cite{Kohatsu06} further considered the optimization problem of large insider trading and derived a characterization theorem for the optimal logarithmic portfolio strategy. They supposed the insiders have some influence on the price of the risky-asset. Thus the dynamic of the risky-asset is also an anticipating SDE. Many other extensions can be found in, for example, \cite{DiNunno06,DiNunno09,Draouil15,Peng21,Escudero21}, which are all based on the forward integral.

Although many insider-trading problems have been studied, the basic theory of forward integrals is not completed since it is difficult to obtain the properties of forward integrals without other theories. Moreover, some conditions of results, like the forward integrability of coefficients of the risky asset $S$, are abstract and difficult to verify. From this point of view, Malliavin calculus provides a natural way to study the properties of forward integrals due to the completeness of its theory and the relationship between the forward integral and the Skorohod integral. \cite{Russo93,Alos98,Leon00, Kohatsu06,Nualart06} studied the basic properties of forward integrals using Malliavin calculus. However, some conditions of these properties are restrictive (see Remarks \ref{introres} and \ref{introres2}) and could be weakened by considering the space ${\cal L}^{1,2,q-}$ (see Definition \ref{ddd}) and the localization technique (see Definition \ref{localtech}), which were introduced by \cite{Nualart06}. Moreover, the most important It\^{o} formula for forward integrals has not been derived by the method of Malliavin calculus. Although \cite{Nualart06} gave a theorem for the It\^{o} formula, the forward integral is defined by Riemann sums and it requires an extra continuous condition which might not be applied to insider trading directly since we suppose the portfolio process $\pi$ to be c\`{a}gl\`{a}d.

Besides insider trading, the robust optimal portfolio under model uncertainty is another important topic of the optimization problem, which has been taken into account by many scholars in recent years. As pointed out by \cite{Chen02}, the risk-based models that constitute the paradigm have well documented empirical failures. A robust optimal strategy is the optimal strategy under the worst-case probability assuming that the ambiguity aversion is considered. In other words, when the investor is ambiguity averse, she might not believe the model is accurate by empirical statistics, which forces her to choose the robust optimal strategy under the model uncertainty. Thus the optimization problem (\ref{intro_equ1}) turns to the following robust optimization problem (see \cite{Chen02})
\begin{equation}\label{intro_equ2}
\begin{aligned}
 \min_{\theta}\max_{\pi}\mathbb{E}_{{\cal Q}^\theta}\left[ U(X^\pi_T)+\int_0^Tg(\theta_s)\mathrm{d}s  \right]=  \max_{\pi}\min_{\theta}\mathbb{E}_{{\cal Q}^\theta}\left[ U(X^\pi_T)+\int_0^Tg(\theta_s)\mathrm{d}s  \right],
\end{aligned}\end{equation}
where ${\cal Q}^{\theta}$ is the prior probability measure to describe the model uncertainty parametrized by $\theta$, and $g$ is viewed as a step adopted to penalize the difference between ${\cal Q}^\theta$ and original reference probability $\mathbb P$. In fact, ${\cal Q}^\theta$ is induced by the density process $\varepsilon^\theta$ (the Dol\'{e}ans-Dade expoential of $\theta$) by the Girsanov theorem. The model uncertainty was designed originally to face the fact that mean rates of return and other drift parameters in stochastic models for risky assets are difficult to estimate. We refer to \cite{Chen02,Maenhout04,Maenhout06,Flor14,Gu20} for further studies. 

When combining model uncertainty with insider trading, (\ref{intro_equ2}) turns to an anticipating stochastic differential game problem. One can not apply the forward integral method (\cite{Biagini05}) directly since there is no useful result with variation to the other controlling process $\theta$. \cite{An13} developed a general stochastic maximum principle for the anticipating stochastic differential game problem by using Malliavin calculus. However, the result could only be applied to controlled It\^{o}-L\'{e}vy processes due to the complexity of the Malliavin derivative. When considering the insider-trading problem under model uncertainty, the controlling process $\theta$ is also supposed to be ${\cal H}_t$-adapted, the reason of which could be found in \cite{Chen02}. Thus the construction (and corresponding variation) of the set of prior probability measures is another difficult point since the Girsanov theorem could not be applied directly to the anticipating integral. 

To the best of our knowledge, only \cite{Danilova10} and \cite{Peng21} considered the insider-trading problem under model uncertainty. However, the insider information filtration are both in the special initial enlargement type (\ref{initialtype}). Moreover, \cite{Danilova10} only considered the special parameter uncertainty, and the robust optimal strategy was not considered since the uncertain parameter could be estimated by the filter algorithm. The method was based on the classical enlargement for filtration technique. \cite{Peng21} considered the robust optimal strategy for a insurer with insider information. However, the ${\cal Q}^\theta$ constructed in their paper is not exact a probability measure when $\theta(y)=\theta(Y)$, and the robust optimal strategy was characterized by a nested linear BSDE, which is very hard to solve.

Enlightened by these works, this paper is devoted to solve the optimization problem of insider trading under model uncertainty, i.e., the robust optimal portfolio for an insider. We use the theory of forward integrals to get the half equations about the optimal strategy, and then transform the problem into an nonaticipative stochastic control problem to obtain the other half equations by the stochastic maximum principle. Then we discuss the two typical cases where the insider is `small' (i.e., her strategy has no influence on the financial market) and `large' (i.e., her strategy has certain influence on the financial market). The main contributions are as follows.

\begin{itemize}
 \item  On the aspect of basic mathematical theory, we improve some properties of the forward integral using the Malliavin calculus, and derive the It\^{o} formula for forward integrals by Skorohod integrals.
\end{itemize}

\begin{itemize}
 \item  We introduce the abstract definition of the set of prior probability measures by the semimartingale decomposition theory, and construct the optimization model with the general insider information and model uncertainty. We first give a half characterization of the robust optimal portfolio strategy with respect to the controlling process $\pi$, and derive the semimartingale property of $W$ with respect to the insider information filtration $\{{\cal H}_t\}_{0\le t\le T}$. Thus the problem turns to classic nonanticipative stochastic differential game problem. Then we give the total characterization of the robust optimal portfolio strategy using the stochastic maximum principle.
 \end{itemize}

\begin{itemize}
 \item  We further give the specific characterization of the robust optimal portfolio when dealing with typical cases. When considering the small insider case and the large insider case, we turn our problem to a (non-nested) linear BSDE by the stochastic maximum principle and to a quadratic BSDE by a combined method, respectively. In particular, when the insider information filtration is of the initial enlargement type, we use the Donsker $\delta$ functional technique in Malliavin calculus to solve these BSDEs with general filtrations, and derive the closed form of the robust optimal portfolio strategy in the small insider case.
\end{itemize}

\begin{itemize}
 \item  We do simulation and Economic analysis based on the analytic expression for the value function and strategy in the continuous case.We find that in all types both insider information and ‘large’ insurer have a positive impact on the utility and ambiguity-aversion has a negative impact on the utility.In order to compare utility gain in different situation, We obtain the critical future time for a trader to balance its loss on robust case ,   that is ,the quality of inside information needed to  balance its loss due to model uncertainty .This critical  point decreases with drift part and increase with volatility.
\end{itemize}

This paper is organized as follows. In Section~\ref{MVcalculus}, we introduce the basic theory of the Malliavin calculus and improve the theory of forward integrals by Malliavin calculus. Moreover, we derive the It\^{o} formula for forward integrals (Theorem \ref{itofor}) by Skorohod integrals. In Section~\ref{Sec:model}, we formulate the optimization problem of insider trading with model uncertainty. We give the first characterization and the total characterization of the robust optimal portfolio in Section~\ref{half_section} and Section~\ref{tt_section}, respectively. We give further characterizations (BSDEs) of the robust optimal portfolio when considering the situation of small insider and large insider in Section~\ref{sec_exampleg_sub1} and Section~\ref{sec_example1}, respectively. We derive the closed form of the robust optimal portfolio in some special cases in Section~\ref{sec_exampleg_sub1}. Moreover, we give the optimal strategy for a large insider in the special case without model uncertainty in Section~\ref{Sub_particular2} for comparison with that under model uncertainty in the next section. The simulation and economic analysis are performed in Section~\ref{numericalsec}. Finally, we summarize our conlusions in Section~\ref{sec:conclusion}.

\section{The forward integral by Malliavin calculus}
\label{MVcalculus}
In this section, we briefly introduce the basic theory of the Skorohod integral in Malliavin calculus (see \cite{Nualart06, Huang00}), and then using it to improve the theory of the forward integral. The main result of this section is that we extend some propositions and obtain the It\^o formula for the forward integral by Malliavin calculus to fit our setting. Moreover, some important details and supplements which may be ignored in Malliavin calculus (\cite{Nualart06}) are provided. 

\subsection{The basic theory of Malliavin calculus}
Consider a filtered probability space $(\Omega, {\cal F}_T, \{{\cal F}_t\}_{0\le t\le T},{\mathbb P})$ and the standard Brownian motion ${ W}=\{{ W}_t\}_{0\le t\le T}$, where $\{{\cal F}_t\}_{0\leq t\leq T}$ is the $\mathbb P$-augmentation of the filtration generated by $ W$, which satisfies the usual condition (see \cite{Karatzas91}). We also denote by $H$ the real Hilbert space $L^2([0,T])$. Then $(\Omega,{\cal F}_T,\mathbb{P};H)$ is an irreducible Gaussian space (see \cite{Huang00}).

We denote by $C^\infty_p({\mathbb R }^n)$ the set of all infinitely continuously differentiable functions $\varphi$ such that $\varphi$ and all of its partial derivatives have
polynomial growth. For a given separable Hilbert space $E$, denote by ${\cal S}(\Omega; E)$ the class of $E$-valued smooth random variables such that $X\in{\cal S}$ has the form
\begin{equation*}
X=\sum_{j=1}^m\varphi_j\left(\int_0^T h_{j,1}(s)\mathrm{d}{ W}_{s},\cdots, \int_0^T h_{j,n_j}(s)\mathrm{d}{ W}_{s}\right)e_j,
\end{equation*}
where $m,n_j\in {\mathbb N_+}$, $\varphi_j\in C^\infty_p({\mathbb R }^{  n_j})$, $h_{j,l}\in H$, and $e_j\in E$ for $l=1, \cdots, n_j$ and $j=1,\cdots,m$. Note that ${\cal S}(\Omega; E)$ is dense in $L^p(\Omega; E)$\footnote{$L^p(\Omega; E):=\{X: \Omega\rightarrow E$ is measurable and $\mathbb {E}\Vert X\Vert_{E}^p<\infty\}$ for $p\ge 1$.} for $p\ge 1$. The Malliavin gradient ${ D}_t$ of the $E$-valued smooth random variable $X$ is defined as the $H \otimes E$-valued random variable ${ D}_tX $ give by
\begin{equation*}
D_t X:=\sum_{j=1}^m\sum_{l=1}^{n_j} \frac{\partial \varphi_j}{\partial x_{l}}\left(\int_0^T h_{j,1}(s)\mathrm{d}{ W}_{s},\cdots, \int_0^T h_{j,n_j}(s)\mathrm{d}{ W}_{s}\right)h_{j,l}(t)\otimes e_j.
\end{equation*}
 For $k=2,3,\cdots$, the $k$-iteration of the operator $D_t$ can be defined in such a way that for $X\in{\cal S}(\Omega;E)$, $D^k_tX$ is a random variable with values in $H^{\otimes k}\otimes E$.

We can check that ${ D}^k_t$ is a closable operator from ${\cal S}(\Omega;E)\subset  L^p(\Omega;E)$ to $ L^p(\Omega; H^{\otimes k}\otimes E)$\footnote{When $p=2$, we have $L^2(\Omega; H^{\otimes k}\otimes E)\cong L^2(\Omega)\otimes H^{\otimes k}\otimes E\cong L^2(\Omega\times [0,T]^k)\otimes E\cong L^2(\Omega\times [0,T]^k;E)$ (see \cite{Huang00}).} for $k\in{\mathbb N}_+$ and $p\ge 1$. Denote by $D^{k,p}(\Omega;E)$ the closure of the class of smooth random variables ${\cal S}(\Omega, E)$ with respect to the graph norm (see \cite{Brezis10})
\begin{equation*}
\begin{aligned}
\| X\|_{D^{k,p}(\Omega;E)}:&=\left[\|X\|_{L^p(\Omega;E)}^p+\sum_{j=1}^k\|{ D}^j_tX\|^p_{L^p(\Omega; H^{\otimes j}\otimes E)} \right]^{\frac{1}{p}}.
\end{aligned}
\end{equation*}
Then $({ D}^1_t,\cdots,D^k_t)'$ is a closed dense operator with dense domain $D^{k,p}(\Omega;E)$, which is a Banach space under the norm $\Vert\cdot\Vert_{D^{k,p}(\Omega;E)}$ and even a Hilbert space when $p=2$. In addition, we define $D^{k,\infty}(\Omega;E):=\bigcap_{p\ge 1}D^{k,p} (\Omega;E)$ and $D^{\infty,\infty}(\Omega;E):=\bigcap_{k\in {\mathbb N}_+}\bigcap_{p\ge 1}D^{k,p}(\Omega;E)$, which are both locally convex space (see \cite{Brezis10, Huang00}).

When $E=\mathbb {R}$, $k=1$ and $p=2$, define by ${\delta }: L^2(\Omega\times [0,T])  \rightarrow L^2(\Omega)$ with domain $\text{Dom}\ { \delta}$ the adjoint of the closed dense operator ${ D}_t:L^2(\Omega)\rightarrow L^2(\Omega\times [0,T])$ that is called the Malliavin divergence operator, or the Skorohod integral. If we denote by $L_a^2(\Omega\times [0,T])$ the set of all ${\cal F}_t$-adapted processes $u\in L^2(\Omega\times [0,T])$, then $L_a^2(\Omega\times [0,T])\subset \text{Dom } \delta$, and $\delta u$ equals the It\^{o} integral $\int_0^Tu_t\mathrm{d}W_t$ for $u\in L_a^2(\Omega\times [0,T])$. Thus we will keep use of the notation $\int_0^T{ u}_t \mathrm{d}{ W}_t:=\delta { u}$ when ${ u}\in \text{Dom } \delta $. 

There are rich properties of $D_t$ and $ \delta$. Some of them are given in Appendix $\ref{appM}$ which might be needed later.

  \subsection{The Skorohod integral}\label{skorohodin}
 When $u$ is Skorohod integrable (i.e., $u\in \text{Dom }\delta$), a natural question is that whether $\int_0^tu_s\mathrm{d}W_s:=\delta (u_s1_{[0,t]}(s))$ makes sense for a fixed $t\in[0,T]$. Unfortunately, $u_s1_{[0,t]}(s)$ is not Skorohod integrable in general. However, since $D^{1,2}(\Omega;H)\subset \text{Dom }\delta$ (see Lemma \ref{boundeddelta}), $\int_0^tu_s\mathrm{d}W_s$ is well-defined for $u\in D^{1,2}(\Omega;H)$ by the chain rule (Lemma \ref{chainrule}), and we can obtain more useful results in some subspaces of $D^{1,2}(\Omega;H)$.
  
  \begin{definition} 
Define by ${\cal L}^{1,2}$ the space $D^{1,2}(\Omega;H)$, which is isomorphic to $  L^2([0,T];D^{1,2}(\Omega)) $ (see \cite{Nualart06}). For every $k\in{\mathbb N}_+$ and any $p\ge 2$,  define by ${\cal L}^{k,p}$ the space $L^p([0,T];D^{k,p}(\Omega))$, which is a subspace of $  D^{k,p}(\Omega;H)$.
  \end{definition}

  \begin{definition} \label{ddd}
 Let $u\in {\cal L}^{1,2}$, and let $q\in [1,2]$. We say that $u\in {\cal L}^{1,2,q-}$ (resp. $u\in {\cal L}^{1,2,q+}$) if there exists a (unique) process in $L^{q}(\Omega\times [0,T])$ denoted by $D^-u$ (resp. $D^+u$) such that
 \begin{equation*}\begin{aligned}
& \lim_{\varepsilon\rightarrow 0^+}\int_0^T\sup_{(s-\varepsilon)\vee 0\le t<s}\mathbb{E}(|D_su_t-(D^-u)_s|^q)=0.\\
 &(\text{resp. } 
  \lim_{\varepsilon\rightarrow 0^+}\int_0^T\sup_{s< t\le (s+\varepsilon)\wedge T}\mathbb{E}(|D_su_t-(D^+u)_s|^q)=0)
\end{aligned} \end{equation*}
In particular, if $u\in  {\cal L}^{1,2,q-}\cap  {\cal L}^{1,2,q+}$, we say that $u\in  {\cal L}^{1,2,q}$, and define $\nabla u:=D^-u+D^+u$ which is also in $L^q(\Omega\times[0,T])$.
 \end{definition}

\begin{remark}\label{earlyth}
In the earlier theory of the Skorohod integral,  the space ${\cal L}^{1,2,C}$ (see \cite{Nualart95, Huang00}), which allows the existence of $(D^-u)_s:=\lim_{\varepsilon\rightarrow 0^+}D_su_{s-\varepsilon}$ and $(D^+u)_s:=\lim_{\varepsilon\rightarrow 0^+}D_su_{s+\varepsilon}$ in $L^2(\Omega)$ uniformly in $s$, is too restrictive. It can not characterize  the convergence in $L^q(\Omega)$ and makes some proofs for the sufficiency difficult. Moreover, we will not suppose the existence of  $D^+u$ in the discussion of the forward integral in Section \ref{forwardin}, which might not be realized for a c\`{a}gl\`{a}d process in financial problems. Thus we will analyze in the more general spaces ${\cal L}^{1,2,q-}$ and ${\cal L}^{1,2,q+}$ instead of ${\cal L}^{1,2,C}$ introduced in \cite{Nualart06}, the second edition of \cite{Nualart95}. 
\end{remark}

  Like the It\^o formula in the classical It\^o theory (see \cite{Karatzas91}), there is also a version of the It\^o formula for the Skorohod integral. However, before the statement, a localization technique is required, which is similar to the method of the local martingale in It\^o theory. 
  
  \begin{definition}\label{localtech}
 If $L $ is a class of random variables (or random fields), we denote by $L_{\text{loc}}$ the set of random variables (or random fields) $X$ such that there exists a sequence $\{(\Omega_n,X_n)\}\subset {\cal F}\times L$ with the following properties:
\begin{itemize}
 \item[(i)] $\Omega_n \uparrow \Omega$, a.s.
\item[(ii)] $X_n=X$ a.s. on $\Omega_n$.
\end{itemize}
Moreover, we can easily check that $L_{\text{loc}}$ is a linear space if $L$ is a linear space.
  \end{definition}

Due to the local properties of $D_t$ and $\delta$ (Lemma \ref{locD} and \ref{locdelta}),  the extensions of $D^k_t: D^{k,p}_{\text{loc}}(\Omega;E)\rightarrow L^p_{\text{loc}}(\Omega;H^{\otimes k}\otimes E)$ ($p\ge 1$) and $\delta:{\cal L}^{1,2}_{\text{loc}} \rightarrow L^2_{\text{loc}}(\Omega) $ are well-defined given that $E$ is a separable Hilbert space.  The localizations for $D^-$, $D^+$ and $\nabla$ are in a similar way.  The next proposition shows that the Skorohod integral is also an extension of the generalized It\^o integral in the sense of localization.

   \begin{proposition}\label{locik}
   \textnormal{(Proposition 1.3.18,  \cite{Nualart06})} 
   Let $u$ be a measurable, ${\cal F}_t$-adapted process such that $\int_0^Tu_t^2\mathrm{d}t<\infty$, a.s. Then $u$ belongs to $ (\text{Dom }\delta)_{\text{loc}}$ and $ \delta u$ is well-defined. Moreover, $\delta u$ coincides with the It\^o integral of $u$ (with respect to the local martingale $W$). Thus we can keep use of the notation $\int_0^Tu_t\mathrm{d}W_t:=\delta u$ without ambiguity when $u\in (\text{Dom } \delta)_{\text{loc}}$ and $\delta u$ is well-defined.   
   \end{proposition}

 Now we can give the It\^o formula for the Skorohod integral.
 
  \begin{theorem}\label{itos}
  \textnormal{(Theorem 3.2.2,  \cite{Nualart06})} 
  Consider a process of the form $X_t=X_0+\int_0^tu_s\mathrm{d}W_s+\int_0^tv_s\mathrm{d}s$, where $X_0\in D^{1,2}_{\text{loc}}(\Omega)$, $u\in ({\cal L}^{2,2} \cap {\cal L}^{1,4})_{\text{loc}}$, and $v\in {\cal L}^{1,2}_{\text{loc}}$. Then $X_t$ is continuous and $X\in {\cal L}^{1,2,2}_{\text{loc}}$ by Lemma \ref{continuous} and \ref{sko2}, respectively. Moreover, if $f\in C^2 (\mathbb{R})$, then $f'(X_t)u_t\in {\cal L}_{\text{loc}}^{1,2}$ and we have 
   \begin{equation}\label{equitos}
\begin{aligned}
f(X_t)=f(X_0)+\int_0^tf'(X_s)\mathrm{d}X_s+\frac{1}{2}\int_0^tf''(X_s)u_s^2\mathrm{d}s+\int_0^tf''(X_s)(D^-X)_su_s\mathrm{d}s.
\end{aligned}
\end{equation}
\end{theorem}

   \subsection{The forward integral}\label{forwardin}
The Skorohod integral process $\int_0^tu_s\mathrm{d}W_s$ is anticipating, i.e., it is not adapted to the filtration ${\cal F}_t$. There is another anticipating integral called the forward integral, which was introduced by \cite{Berger82} and defined by \cite{Russo93}. This type of integral has been studied before and applied to insider trading in financial mathematics (see \cite{Russo00, Biagini05}). However, the sufficiency of the forward integrability and some related properties may be hard to obtain without the help of Malliavin calculus (see \cite{Alos98, Leon00, Kohatsu06}).  Since some results in the above literature are restrictive, we will study the forward integral by Malliavin calculus here completely. All proofs of our results in this subsection are given in Appendix \ref{Appp}.
  
    \begin{definition}
Let $u\in L^2(\Omega\times [0,T])$. The forward integral of $u$ is defined by
\begin{equation} 
\begin{aligned}
\int_0^T u_t\mathrm{d}^-W_t:=\lim_{\varepsilon\rightarrow 0^+}\varepsilon^{-1}\int_0^Tu_t\left(W_{(t+\varepsilon)\wedge T}-W_t\right)\mathrm{d}t,
\end{aligned}
\end{equation}
if the limit exists in probability, in which case $u$ is called forward integrable and we write $u\in\text{Dom } \delta^-$. If the limit exists also in $L^p(\Omega)$, we write $u\in\text{Dom}_p\  \delta^-$.
  \end{definition}

 \begin{remark}\label{extensionoffor}
 The forward integral is also an extension of the It\^{o} integral. In other words, if there is a filtration $\{{\cal E}_t\}_{0\le t\le T}$ satisfying the usual condition such that ${\cal E}_t\supset {\cal F}_t$ and $W$ is a semimartingale with respet to $\{{\cal E}_t\}$, $t\in[0,T]$, then we have $\int_0^Tu_t\mathrm{d}^-W_t=\int_0^Tu_t\mathrm{d}W_t$ for every ${\cal E}_t$-adapted process $u$ such that $u$ is It\^{o} integrable with respect to $W$. We refer to \cite{DiNunno09} for the proof.
 \end{remark}

  Notice that, like the Skorohod integral, the forward integrability of $u_s1_{[0,t]}(s)$ for $t\in[0,T]$ can not be deduced from that of $u$, which might be ignored in some literature.   However, by Malliavin calculus, we can give the sufficient condition for the above problem and the relationship between the Skorohod integral and the forward integral as the following two propositions, which has not been proved in our settings.

     \begin{proposition}\label{multipf1}
 Let $u\in {\cal L}^{1,2,1-}$. Then for all $t\in[0,T]$, we have $u_s1_{[0,t]}(s)\in \text{Dom}_1  \ \delta^-$ and  
 \begin{equation} \label{equ_for1}
\begin{aligned}
\int_0^t u_s\mathrm{d}^-W_s=\int_0^tu_s\mathrm{d}W_s+\int_0^t(D^-u)_s\mathrm{d}s.
\end{aligned}
\end{equation}
   \end{proposition}

\begin{remark}\label{introres}
In \cite{Kohatsu06}, the condition $\lim_{\varepsilon\rightarrow0^+}\frac{1}{\varepsilon}\int_{t-\varepsilon}^tu_s\mathrm{d}s=u_t$ in ${\cal L}^{1,2}$ which makes (\ref{equ_for1}) hold is surplus (see Lemma \ref{conv_1}), and the use of $D_{t^+}$ is restrictive by Remark \ref{earlyth}. In \cite{Alos98}, the condition $u\in {\cal L}^{F}$ requires the second derivative of $u$, whereas we do not. In \cite{Nualart06}, the forward integral is defined by Riemann sums, and (\ref{equ_for1}) requires an extra continuous condition which might not be applied to insider trading theory directly when we consider the c\`{a}gl\`{a}d process.
\end{remark}
   
\begin{remark}
 Proposition \ref{multipf1} also shows that the forward integral can be extended to a linear operator $\delta^-$ from ${\cal L}^{1,2,1-}_{\text{loc}}$ into $L^1_{\text{loc}}(\Omega)$.
 \end{remark}

     \begin{proposition}\label{multipf2}
Let $u$ be a process in $  {\cal L}^{1,2,2-}$ and be $L^2$-bounded. Consider an ${\cal F}_t$-adapted process $\sigma\in {\cal L}^{1,2}$, which is $L^2$-bounded and left-continuous in the norm $L^2(\Omega)$. Also assume that $\sigma$ and $D_s\sigma_t$ are bounded. Then $u\sigma\in {\cal L}^{1,2,1-}$,  and for all $t\in [0,T]$,
 \begin{equation} \label{equ_for2}
\begin{aligned}
\int_0^t u_s\sigma_s\mathrm{d}^-W_s=\int_0^tu_s\sigma_s\mathrm{d}W_s+\int_0^t(D^-u)_s\sigma_s\mathrm{d}s.
\end{aligned}
\end{equation}
   \end{proposition}
   
   \begin{remark}\label{introres2}
In \cite{Leon00}, the condition which makes (\ref{equ_for2}) hold requires the introduction of some other spaces and norms, and the proof is too complicated. Moreover, the space ${\cal L}^{1,2,C}$ is restrictive. 
   \end{remark}
    
       The It\^{o} formula for the forward integral was first proved in \cite{Russo00} without Malliavin calculus. Here we use the It\^{o} formula for the Skorohod integral (Theorem \ref{itos}) to obtain it, which can be viewed as an extension of \cite{Russo00}. Denote by ${\cal L}^f$ the linear space of processes $u\in {\cal L}^{2,2}\cap {\cal L}^{1,4}\cap {\cal L}^{1,2,1-}$ left-continuous in $L^2(\Omega)$, $L^2$-bounded, and such that $(D^-u)\in {\cal L}^{1,2}$. Then we have the following theorem.

     \begin{theorem}\label{itofor}
  Consider a process of the form $X_t=X_0+\int_0^tu_s\mathrm{d}^-W_s+\int_0^tv_s\mathrm{d}s$, where $X_0\in D^{1,2}_{\text{loc}}(\Omega)$, $u\in {\cal L}^{f}_{\text{loc}}$, and $v\in {\cal L}^{1,2}_{\text{loc}}$. Then $X_t$ is continuous and $X\in {\cal L}^{1,2,2}_{\text{loc}}$. Moreover, if $f\in C^2 (\mathbb{R})$, then $f'(X_t)u_t\in {\cal L}_{\text{loc}}^{1,2,1^-}$ and we have 
   \begin{equation}\label{equitos}
\begin{aligned}
f(X_t)=f(X_0)+\int_0^tf'(X_s)u_s\mathrm{d}^-W_s+\int_0^tf'(X_s)v_s\mathrm{d}s+\frac{1}{2}\int_0^tf''(X_s)u_s^2\mathrm{d}s.
\end{aligned}
\end{equation}
 \end{theorem}

\section{Model formulation}
 \label{Sec:model}
We assume that all uncertainties come from the fixed filtered probability space $(\Omega,{\cal F},\{{\cal F}_t\}_{t\ge 0},{\mathbb P})$ and a standard Brownian motion $W$ in our model, where $\{{\cal F}_t\}_{t\ge0}$ is the $\mathbb{P}$-augmentation of  the filtration generated by $W$ and ${\cal F}={\cal F}_\infty$. Fix a terminal time $T>0$. Suppose all filtrations given in this section satisfy the usual condition. 

Consider an investor who can invest in the  financial market containing a risk-free asset (bond) $B$ and a risky asset (stock) $S$. The price processes of the two assets are governed by the following anticipating SDEs
\begin{eqnarray} 
  \left\{ \begin{aligned}&\mathrm{d}B_t=r_tB_t\mathrm{d}t,\quad 0\le t\le T,
  \\&   \mathrm{d}S_t=   \mu(t,\pi_t)S_{t}\mathrm{d}t+\sigma_tS_{t}\mathrm{d}^- W_t  , \quad 0\le t\le T,\end{aligned} \right.   
  \label{fin-m}
\end{eqnarray}
with constant initial values $1$ and $S_0>0$, respectively. Here, the coefficients $r_t, {\mu}(t,x)$ for fixed $x\in\mathbb{R}$, and $\sigma_t$, $t\in [0,T]$, are all ${\cal F}_t$-adapted measurable stochastic processes, and $\mu(t,\cdot)$ is $C^1$ for every $t\in[0,T]$. Assume the investor is a large investor and has access to insider information  characterized by anothet filtration $\{{\cal H}_t\}_{0\le t\le T}$ such that
\begin{equation}
{\cal F}_t \subset {\cal H}_t, \quad0\le t\le T.
\end{equation}
 The portfolio strategy she takes could influence the mean rate of return $\mu$ of the risky asset. Thus, $\mu$ partly depends on her portfolio strategy $\pi$ (see \cite{Kohatsu06, DiNunno09}). Here, $\pi_t$ is defined as an ${\cal H}_t$-adapted c\`{a}gl\`{a}d process satisfying $\pi\in {\cal L}^{1,2,2-}$ and is $L^2$-bounded. It represents the proportion of the investor's total wealth $X_t$ invested in the risky asset $S_t$ at time $t$. Since $\mu(t,\pi_t)$ is not adapted to $\{{\cal F}_t\}_{0\le t\le T}$, the stochastic integral in (\ref{fin-m}) should be interpreted as the forward integral.

 We make some assumptions on the coefficients as follows:
 
\begin{itemize}
 \item $r\in {\cal L}^{1,2}$. For each portfolio strategy $\pi$, $\mu(\cdot,\pi)-\frac{1}{2}\sigma^2\in {\cal L}^{1,2}$. $\sigma\ge \epsilon>0$ for some positive constant $\epsilon $, and $\sigma\in {\cal L}^{f}$;
\end{itemize}
 \begin{itemize}
 \item  $\sigma$ and $D_s\sigma_t$ are bounded.
  \end{itemize}

Under the above conditions, we can solve the anticipating SDEs (\ref{fin-m}) by using the It\^{o} formula for forward integrals (Theorem \ref{itofor}). Moreover, the process of the risky asset $S$ is given by
 \begin{equation}
 \begin{aligned}
S_t=S_0\exp \Bigg\{   &   \int_0^t  \left(  \mu(s,\pi_s)-\frac{1}{2} \sigma^2_s \right) \mathrm{d}s+\int_0^t \sigma_s \mathrm{d}W_s \Bigg\}, \quad 0\le t\le T,
\end{aligned}
 \end{equation}
 which is no more an ${\cal F}_t$-semimartingale, but an ${\cal H}_t$-adapted process.

Note that the portfolio strategy $\pi$ of the investor can take negative values, which is to be interpreted as short-selling the risky asset. Then her wealth process $X^\pi$ corresponding to $\pi$ is governed by the following anticipating SDE (see \cite{Nualart06}):
 \begin{equation}\label{wealthsde}
 \begin{aligned}
\mathrm{d}X^\pi_t=&\left[r_t+(\mu(t,\pi_t)-r_t)\pi_t \right]X^\pi_{t}\mathrm{d}t+ \sigma_t\pi_t X^\pi_{t}\mathrm{d}^-W_t, \quad 0\le t\le T,
\end{aligned}
 \end{equation}
with constant initial value $X_0>0$. We can apply the It\^{o} formula for forward integrals to solve the anticipating SDE (\ref{wealthsde}). Before that, we impose the following admissible conditions on $\pi$.

\begin{definition}\label{admiss}
We define ${\cal A}_1$ as the set of all portfolio strategies $ \pi $ satisfying the following conditions:
  \begin{itemize}
 \item[(i)] $\sigma\pi\in{\cal L}^{f}$; 
 \end{itemize}
 \begin{itemize}
\item[(ii)]  $(\mu(\cdot,\pi)-r)\pi-\frac{1}{2}\sigma^2\pi^2\in{\cal L}^{1,2}$;
\end{itemize}
 \begin{itemize}
\item[(iii)]  $\int_0^T|r_t+(\mu(t,\pi_t)-r_t)\pi_t|\mathrm{d}t+\int_0^T|\sigma_t\pi_t|^2\mathrm{d}t<\infty$.
\end{itemize}
\end{definition}

Let $u\in{\cal A}_1$. By Theorem \ref{itofor}, the solution of (\ref{wealthsde}) is given by
 \begin{equation}\label{equofxt}
 \begin{aligned}
X^\pi_t=X_0\exp \Bigg\{   &   \int_0^t  \left[ r_s+(\mu(s,\pi_s)-r_s)\pi_s-\frac{1}{2}\sigma_s^2 \pi_s^2 \right] \mathrm{d}s+\int_0^t \sigma_s\pi_s \mathrm{d}^-W_s \Bigg\}, \quad 0\le t\le T.
\end{aligned}
 \end{equation}

Consider a model uncertainty setup. Assume that the investor is ambiguity averse, implying that she is concerned about the accuracy of statistical estimation, and possible misspecification errors. Thus, a family of parametrized prior probability measures $\{{{\cal Q}}^\theta\}$ equivalent to the original probability measure $\mathbb{P}$ is assumed to be exist in the real world. However, since the investor has insider information filtration $\{{\cal H}_t\}$ under which $W$ might not be a semimartingale, a generalization for the construction of $\{{{\cal Q}}^\theta\}$ need to be considered by means of the forward integral. 

\begin{definition}\label{admiss_1}
We define ${\cal A}_2$ as the set of all ${\cal H}_t$-adapted c\`{a}gl\`{a}d processes $\theta_t $, $t\in [0,T]$, satisfying the following conditions:
\begin{itemize}
 \item[(i)] $\theta \in{\cal L}^{1,2,1-}$;
 \end{itemize}
 \begin{itemize}
\item[(ii)]   $\int_0^t\theta_s\mathrm{d}^-W_s$ is a continuous ${\cal H}_t$-semimartingale, the local martingale part (in the canonical decomposition) of which satisfies the Novikov condition (see \cite{Karatzas91}).
\end{itemize}
\end{definition}

For $\theta\in {\cal A}_2$, the Dol\'{e}ans-Dade exponential $\varepsilon^\theta_t$ is the unique ${\cal H}_t$-martingale  governed by 
 \begin{equation}\label{theta_sde}
 \begin{aligned}
 \varepsilon^\theta_t=1+\varepsilon^\theta_{t}\left(  \int_0^t\theta_s\mathrm{d}^-W_s \right)^M, \quad 0\le t\le T,
\end{aligned}
 \end{equation}
where $(\cdot)^M$ denotes the local martingale part of a continuous semimartingale. Thus, we have $\varepsilon^\theta_T>0$ and $\int_\Omega\varepsilon^\theta_T\mathrm{d}\mathbb{P}=1$, which induces a probability ${{\cal Q}}^\theta$ equivalent to $\mathbb{P}$ such that $ \frac{\mathrm{d}{\cal Q}^\theta}{\mathrm{d}\mathbb{P}} =\varepsilon^\theta_T$. Then all such ${\cal Q}^\theta$ form a set of prior probability measures $\{{\cal Q}^\theta\}_{\theta\in{\cal A}_2}$.

Taking into account the extra insider information and model uncertainty, the optimization problem for the investor can be formulated as a (zero-sum) anticipating stochastic differential game. In other words, we need to solve the following problem.

\begin{problem}\label{sdg}
 Select a pair $(\pi^*,\theta^*)\in {\cal A}_1'\times {\cal A}_2'$ (see Definition \ref{admiss12} below) such that
\begin{equation}\label{max_u}
 V:=J(\pi^*,\theta^*)= \inf_{\theta\in{\cal A}_2'}\sup_{\pi\in {\cal A}_1'}  J(\pi,\theta)=\sup_{\pi\in {\cal A}_1'} \inf_{\theta\in{\cal A}_2'} J(\pi,\theta) ,
\end{equation}
where the performance function is given by
\begin{equation*}
J(\pi,\theta):=\mathbb{E}_{{\cal Q}^\theta}\left[   \ln X_T^\pi +\int_0^Tg(\theta_s)\mathrm{d}s\right]=\mathbb{E}\left[  \varepsilon_T^\theta \ln X_T^\pi +\int_0^T \varepsilon_s^\theta g(\theta_s)\mathrm{d}s\right],
\end{equation*}
 the penalty function $g: \mathbb{R}\rightarrow \mathbb{R}$ is a Fr\'{e}chet differentiable convex function. We call $V$ the value (or the optimal expected utility under the worst-case probability) of Problem \ref{sdg}.
\end{problem}

\begin{definition}\label{admiss12}
Define ${\cal A}_1'$ as the subset of $ {\cal A}_1$ such that $\mathbb{E} |\ln X_T^\pi|^2<\infty$ for all $u\in{\cal A}_1'$. Define ${\cal A}_2'$ as the subset of ${\cal A}_2$ such that $\mathbb{E}\left[ |{\varepsilon_T^\theta}|^2  + \int_0^T|g(\theta_s)|^2\mathrm{d}s \right]<\infty$ for all $\theta\in {\cal A}_2'$.
\end{definition}
 
 \begin{remark}\label{aremarkofl2m}
For $\theta\in{\cal A}_2'$, $\left|\varepsilon_t^\theta\right|^2 $ is an ${\cal H}_t$-submartingale and $\mathbb{E} |\varepsilon_t^\theta|^2\le \mathbb{E} |\varepsilon_T^\theta |^2<\infty$ for all $t\in[0,T]$ by the Jensen inequality. 
 \end{remark}

\section{A half characterization of the robust optimal portfolio}
 \label{half_section}

 \begin{assumption}\label{assump1}
If $(\pi^*,\theta^*)\in {\cal A}_1'\times {\cal A}_2'$ is optimal for Problem \ref{sdg}, then for all bounded $\alpha \in {\cal A}_1' $, there exists some $\delta>0$ such that $ \pi^*+y\alpha\in {\cal A}_1' $ for all $|y|<\delta$. Moreover, the following family of random variables 
  \begin{equation*}
 \begin{aligned}
\left \{  \varepsilon^{\theta^* }_T (X_T^{\pi^*+y\alpha})^{-1} \frac{\mathrm{d}}{\mathrm{d}y}X_T^{\pi^*+y\alpha}   \right\}_{y\in(-\delta,\delta)}
 \end{aligned}
 \end{equation*}
is $\mathbb{P}$-uniformly integrable, where $\frac{\mathrm{d}}{\mathrm{d}y}$ means that $\frac{\mathrm{d}}{\mathrm{d}y}X_T^{\pi^*+y\alpha}$ exists and the interchange of differentiation and integral with respect to $\ln X_T^{\pi^*+y\alpha}$ in (\ref{equofxt}) is justified.
 \end{assumption}
 
  \begin{assumption}\label{assump2}
  Let $\alpha_s= \vartheta  1_{(t,t+h]}(s) $, $0\le s\le T$, for fixed $0\le t<t+h\le T$, where $\vartheta$ is an ${\cal H}_t$-measurable bounded random variable in $D^{\infty,\infty}(\Omega)$. Then $\alpha \in{\cal A}_1'$.
  \end{assumption}
  
  \begin{theorem}\label{mainth1}
Suppose $(\pi^*,\theta^*)\in {\cal A}_1'\times {\cal A}_2'$ is optimal for Problem \ref{sdg} under Assumptions \ref{assump1} and \ref{assump2}. Then the following stochastic process
    \begin{equation}\label{equ1ofth1}
 \begin{aligned}
m_t^{\pi^*}:=&\int_0^t   \left(\mu(s,\pi_s^*)-r_s+\frac{\partial}{\partial x}\mu(s,\pi^*_s)\pi^*_s-\sigma^2_s\pi^*_s \right)\mathrm{d}s+\int_0^t \sigma_s\mathrm{d}W_s,\quad 0\le t\le T,
\end{aligned}
 \end{equation}
is an $({\cal H}_t,{{\cal Q}}^{\theta^*})$-martingale.
  \end{theorem}
  
 \begin{proof}
 Suppose that the pair $(\pi^*,\theta^*)\in {\cal A}_1'\times{\cal A}_2'$ is optimal. Then for any bounded $\alpha \in {\cal A}_1' $ and $|y|<\delta$, we have $J(\pi^*+y\alpha,\theta^*)\le J(\pi^*,\theta^*) $, which implies that $y=0$ is a maximum point of the function $y\mapsto J(\pi^*+y\alpha,\theta^*)$. Thus, we have $\frac{\mathrm{d}}{\mathrm{d}y}J(\pi^*+y\alpha,\theta^*){|}_{y=0}=0$ once the differentiability is established. Thanks to Assumption \ref{assump1}, we can deduce by Proposition \ref{multipf2} that 
  \begin{equation}\label{euq_bianfen_1}
 \begin{aligned}
&\frac{\mathrm{d}}{\mathrm{d}y}J(\pi^*+y\alpha,\theta^*){|}_{y=0}\\
=&\mathbb{E}\Bigg\{  \varepsilon_T^{\theta^*} 
\Bigg[\int_0^T  \alpha_s\left(\mu(s,\pi_s^*)-r_s+\frac{\partial}{\partial x}\mu(s,\pi^*_s)\pi^*_s-\sigma^2_s\pi^*_s\right)\mathrm{d}s+\int_0^T\alpha_s\sigma_s\mathrm{d}^-W_s \Bigg]
\Bigg\}
\\=&\mathbb{E}\Bigg\{  \varepsilon_T^{\theta^*} 
\Bigg[\int_0^T  \alpha_s\left(\mu(s,\pi_s^*)-r_s+\frac{\partial}{\partial x}\mu(s,\pi^*_s)\pi^*_s-\sigma^2_s\pi^*_s\right)\mathrm{d}s+\int_0^T\alpha_s\sigma_s\mathrm{d}W_s +\int_0^T (D^-\alpha)_s\sigma_s\mathrm{d}s \Bigg]
\Bigg\}
\\=&0.
 \end{aligned}
 \end{equation}
 Now let us fix $0\le t<t+h\le T$. By Assumption \ref{assump2}, we can choose $\alpha \in{\cal A}_1'$ of the form
   \begin{equation*}
 \begin{aligned}
 \alpha_s=\vartheta 1_{(t,t+h]}(s),\quad 0\le s\le T,
   \end{aligned}
 \end{equation*}
 where $\vartheta\in D^{\infty,\infty}(\Omega)$ is an ${\cal H}_t$-measurable bounded random variable. Then we have $(D^-\vartheta 1_{(t,t+h]})_s=D_s\vartheta 1_{(t,t+h]}(s)$. Combining (\ref{euq_bianfen_1}) with Lemma \ref{multidelta} yields
    \begin{equation*}
 \begin{aligned}
0=\mathbb{E}_{{\cal Q}^{\theta^*}}\Bigg\{  \vartheta
\Bigg[&\int_t^{t+h}   \left(\mu(s,\pi_s^*)-r_s+\frac{\partial}{\partial x}\mu(s,\pi^*_s)\pi^*_s-\sigma^2_s\pi^*_s\right)\mathrm{d}s+\int_t^{t+h}\sigma_s\mathrm{d}W_s.
\Bigg]
\Bigg\}
   \end{aligned}
 \end{equation*}
 Since this holds for all such $\vartheta$,  we can conclude that
     \begin{equation*}
 \begin{aligned}
0=\mathbb{E}_{{\cal Q}^{\theta^*}}\left[ m^{\pi^*}_{t+h}-m^{\pi^*}_t\big{|}{\cal H}_t\right].
   \end{aligned}
 \end{equation*}
 Hence, $m^{\pi^*}_t$, $t\in[0,T]$, is an ${\cal H}_t$-martingale under the probability measure ${\cal Q}^{\theta^*}$. 
 \end{proof}

Moreover, we obtain the following result under the original probability measure $\mathbb P$. Unless otherwise stated, all statements are back to $\mathbb{P}$ from now on.

  \begin{theorem}\label{mainth2}
 Suppose $(\pi^*,\theta^*)\in {\cal A}_1'\times {\cal A}_2'$ is optimal for Problem \ref{sdg} under Assumptions \ref{assump1} and \ref{assump2}. Then the following stochastic process
  \begin{equation*}
 \begin{aligned}
\hat m_t^{\pi^*,\theta^*}:=m^{\pi^*}_t-\int_0^t \varepsilon^{\theta^*}_s\mathrm{d}\big\langle(\varepsilon^{\theta^*})^{-1}, m^{\pi^*}\big\rangle_s,\quad 0\le t\le T,
\end{aligned}
 \end{equation*}
  is an ${\cal H}_t$-local martingales. Here, $\langle \cdot,\cdot\rangle$ represents the covariance process (see \cite{Karatzas91}), and $m^{\pi^*}_t$ is given in Theorem \ref{mainth1}.
  \end{theorem}

\begin{proof}
If $(\pi^*,\theta^*)\in{\cal A}_1'\times {\cal A}_2'$ is optimal, then by Theorem \ref{mainth1} we know that $m^{\pi^*}_t$, $t\in [0,T]$, is an $({\cal H}_t,{\cal Q}^{\theta^*})$-martingale. The conclusion is an immediate result from the Girsanov theorem (see \cite{Protter05}).
 \end{proof}
 
 \begin{remark}
 Note that $\big\langle(\varepsilon^{\theta^*})^{-1}, m^{\pi^*}\big\rangle_t$ is an absolutely continuous process since the quadratic variation of $m_t^{\pi^*}$ is absolutely continuous. In fact, we have
 \begin{equation*}
\big\langle m^{\pi^*}\big\rangle_t=\int_0^t\sigma^2_s\mathrm{d}s, \quad 0\le t\le T.
 \end{equation*}
 \end{remark}

Further, since $\hat m^{\pi^*,\theta^*}_t$, $t\in [0,T]$, is an ${\cal H}_t$-local martingale, we can deduce from (\ref{equ1ofth1}) that $ \int_0^t \sigma_s\mathrm{d}W_s$ is a continuous ${\cal H}_t$-semimartingale. Using the fact that 
\begin{equation*}\begin{aligned}
  \int_0^t \sigma_s^{-1}\mathrm{d}\hat m_s^{\pi^*,\theta^*} =&W_t+\int_0^t   \sigma_s^{-1}\left(\mu(s,\pi_s^*)-r_s+\frac{\partial}{\partial x}\mu(s,\pi^*_s)\pi^*_s-\sigma^2_s\pi^*_s \right)\mathrm{d}s\\
  &-\int_0^t \sigma^{-1}_s\varepsilon^{\theta^*}_s\mathrm{d}\big\langle(\varepsilon^{\theta^*})^{-1}, m^{\pi^*}\big\rangle_s
   \end{aligned}\end{equation*} 
   and $\langle  \hat m^{\pi^*,\theta^*}\rangle_t= \left\langle  \int_0^\cdot \sigma_s\mathrm{d}W_s\right\rangle_t=\int_0^t\sigma_s^2\mathrm{d}s$, we have $\langle W\rangle_t=t$. Thus,  by the L\'{e}vy theorem (see \cite{Karatzas91}), the canonical decomposition of the continuous ${\cal H}_t$-semimartingale $W_t$ can be given as $W_t=W_{{\cal H}}(t)+\int_0^t\phi_s\mathrm{d}s$, where $W_{{\cal H}}(t)$ is an ${\cal H}_t$-Brownian motion and $\phi_t$ is a measurable ${\cal H}_t$-adapted process. Moreover, by Remark \ref{extensionoffor} we have $\int_0^t\sigma_s\mathrm{d}W_s=\int_0^t\sigma_s\mathrm{d}W_{{\cal H}}(s)+\int_0^t\sigma_s\phi_s\mathrm{d}s$. In summary, we have the following theorem.

    \begin{theorem}\label{mainth3}
  Suppose $(\pi^*,\theta^*)\in {\cal A}_1'\times{\cal A}_2'$ is optimal for Problem \ref{sdg} under Assumptions \ref{assump1} and \ref{assump2}. Then we have the following decomposition 
  \begin{equation} \label{decompofW}
 \begin{aligned}
W_t=W_{\cal H}(t)+\int_0^t\phi_s\mathrm{d}s,  \quad 0\le t\le T,
\end{aligned}
 \end{equation}
 where $W_{\cal H}(t)$ is an ${\cal H}_t$-Brownian motion, $\phi_t$ is a measurable ${\cal H}_t$-adapted process satisfying $\int_0^T|\phi_t|\mathrm{d}t<\infty$. Moreover, by the uniqueness of the canonical decomposition of a continuous semimartingale, $\pi^*$ solves the following equation  
   \begin{equation}\label{halfequ_1}
 \begin{aligned}
0
= &\int_0^t   \left(\mu(s,\pi_s^*)-r_s+\frac{\partial}{\partial x}\mu(s,\pi^*_s)\pi^*_s-\sigma^2_s\pi^*_s \right)\mathrm{d}s  +\int_0^t\sigma_s\phi_s\mathrm{d}s\\
&-\int_0^t \varepsilon^{\theta^*}_s\mathrm{d}\big\langle(\varepsilon^{\theta^*})^{-1}, m^{\pi^*}\big\rangle_s,\quad 0\le t\le T.
\end{aligned}
 \end{equation}
    \end{theorem}

 Further, by Theorem \ref{mainth3} and Remark \ref{extensionoffor}, the dynamic of the ${\cal H}_t$-martingale $\varepsilon^\theta_t$ (see (\ref{theta_sde})) can be rewritten as  
  \begin{equation} \label{half_equ2}
 \begin{aligned}
 \varepsilon^\theta_t=1+\varepsilon^\theta_{t}\Bigg( & \int_0^t\theta_s\mathrm{d}W_{\cal H}(s) \Bigg) ,\quad 0\le t\le T,
\end{aligned}
 \end{equation}
 for $ \theta \in{\cal A}_2'$. By the It\^{o} formula for It\^{o} integrals (see \cite{Karatzas91}), we have
  \begin{equation} \label{etominus1}
 \begin{aligned}
 (\varepsilon_t^{\theta})^{-1}&=1+\int_0^t(\varepsilon_s^{\theta})^{-1}\theta_s^2  \mathrm{d}s  -\int_0^t (\varepsilon_s^{\theta})^{-1}\theta_s\mathrm{d}W_{\cal H}(s).
 \end{aligned}
 \end{equation}
  
 For the optimal pair $(\pi^*,\theta^*)$, by Theorem \ref{mainth3} we can easily calculate the covariation process of $(\varepsilon^{\theta^*}_t)^{-1}$ and $m^{\pi^*}_t $ as follows
   \begin{equation} \label{reasonfor1}
 \begin{aligned}
 \big\langle (\varepsilon^{\theta^*})^{-1},m^{\pi^*} \big\rangle_t
=-\int_{0}^t (\varepsilon_s^{\theta^*})^{-1}\theta_s^*\sigma_s\mathrm{d}s.
 \end{aligned}
 \end{equation}
Then we obtain the following theorem by substituting (\ref{reasonfor1}) into (\ref{halfequ_1}) in Theorem \ref{mainth3}.
  
 \begin{theorem}\label{mainth6}
Suppose $(\pi^*,\theta^*)\in {\cal A}_1'\times{\cal A}_2'$ is optimal for Problem \ref{sdg} under Assumptions \ref{assump1} and \ref{assump2}. Then $\pi^*$ solves the following equation
    \begin{equation}\label{halfequ_21}
 \begin{aligned}
0
= & \mu(t,\pi_t^*)-r_t+\frac{\partial}{\partial x}\mu(t,\pi^*_t)\pi^*_t-\sigma^2_t\pi^*_t  +\sigma_t\phi_t+\sigma_t\theta_t^* ,\quad 0\le t\le T.
\end{aligned}
 \end{equation}
 \end{theorem}

\section{A total characterization of the robust optimal portfolio}
\label{tt_section}

 In the previous section, we give the characterization of $\pi^*$ for the optimal pair $(\pi^*,\theta^*)$ by using the maximality of $J({\pi^*},{\theta^*})$ with respect to $\pi$. Thus, we obtain the relationship between ${\pi^*}$ and ${\theta^*}$ (see the equation (\ref{halfequ_21})). However, we have not used the minimality of $J({\pi^*},{\theta^*})$ with respect to $\theta$. Thus, we need the other half characterization of ${\theta^*}$.
 
It is very difficult to give a characterization of ${\theta^*}$ directly due to the complexity of the other half controlled process $\varepsilon^\theta_t$ (see (\ref{theta_sde})). Fortunately, under Assumptions \ref{assump1} and \ref{assump2}, we get the decomposition $W_t=W_{\cal H}(t)+\int_0^t\phi_s\mathrm{d}s$ with respect to the filtration $\{{\cal H}_t\}$ by Theorem \ref{mainth3}. Thus, we have a better expression of $\varepsilon^\theta_t$ for $\theta \in{\cal A}_2'$ (see (\ref{half_equ2})). Moreover, we can also rewrite the dynamic of the wealth process $X^\pi_t$ (see (\ref{wealthsde})) as follows:
 \begin{equation}\label{wealthsde2}
 \begin{aligned}
 {\mathrm{d}X^\pi_t}=\left[r_t+(\mu(t,\pi_t)-r_t)\pi_t +\sigma_t\pi_t \phi_t \right]X^\pi_{t}\mathrm{d}t+  \sigma_t\pi_t X^\pi_{t}\mathrm{d}W_{\cal H}(t), \quad 0\le t\le T,
\end{aligned}
 \end{equation}

 Since (\ref{half_equ2}) and (\ref{wealthsde2}) can be viewed as the classical SDEs with respect to the ${\cal H}_t$-Brownian motion $W_{\cal H}(t)$, Problem \ref{sdg} turns to a nonanticipating stochastic differential game problem with respect to the filtration $\{{\cal H}_t\}_{0\le t\le T}$. Thus, we can use the stochastic maximum principle to solve our problem. 
   
We make the following assumptions before our procedure.
 
 \begin{assumption}\label{assump3}
If $({\pi^*},{\theta^*})\in {\cal A}_1'\times {\cal A}_2'$ is optimal for Problem \ref{sdg}, then for all bounded $  \beta \in {\cal A}_2'$, there exists some $\delta>0$ such that $ {\theta^*}+y\beta \in   {\cal A}_2'$ for all $|y|<\delta$. Moreover, the following family of random variables 
  \begin{equation*}
 \begin{aligned}
\left \{  \frac{\mathrm{d}}{\mathrm{d}y}\varepsilon_T^{{\theta^*}+y\beta}\ln X_T^{{\pi^*} } \right\}_{y\in(-\delta,\delta)}
 \end{aligned}
 \end{equation*}
is $\mathbb{P}$-uniformly integrable, and the following family of random fields
  \begin{equation*}
 \begin{aligned}
\left \{  \frac{\mathrm{d}}{\mathrm{d}y}\varepsilon^{{\theta^*}+y\beta}_tg({\theta^*}+y\beta) +\varepsilon_t^{{\theta^*}+y\beta} g'({\theta^*}+y\beta)\beta \right\}_{y\in(-\delta,\delta)}
 \end{aligned}
 \end{equation*}
is $\mathrm{m} \times \mathbb{P}$-uniformly integrable, where $\mathrm{m}$ is the Borel-Lebesgue measure on $[0,T]$, and $\frac{\mathrm{d}}{\mathrm{d}y}$ means that $\frac{\mathrm{d}}{\mathrm{d}y}\varepsilon_t^{{\theta^*}+y\beta}$ exists. 
 \end{assumption}
 
  \begin{assumption}\label{assump3_2}
If $({\pi^*},{\theta^*})\in {\cal A}_1'\times {\cal A}_2'$ is optimal for Problem \ref{sdg} under Assumptions \ref{assump1}, \ref{assump2} and \ref{assump3}, then for all bounded $ (\alpha , \beta) \in  {\cal A}_1'\times{\cal A}_2'$, we can define $\tilde\psi^{{\pi^*}}_t:=\frac{\mathrm{d}}{\mathrm{d}y}X_t^{\pi^*+y\alpha}{|}_{y=0}$ and $\psi^{{\theta^*}}_t:=\frac{\mathrm{d}}{\mathrm{d}y}\varepsilon_t^{\theta^*+y\beta}{|}_{y=0}$ by Assumptions \ref{assump1} and \ref{assump3}. Assume the following SDEs hold:
 \begin{eqnarray*}
  \left\{ \begin{aligned}&\mathrm{d}\tilde\psi^{{\pi^*}}_t= \Bigg[  \frac{\partial}{\partial x}\mu(t,\pi^*_t) \pi^*_t\alpha_t+(\mu(t,\pi^*_t)-r_t)\alpha_t+\sigma_t\phi_t\alpha_t\Bigg]X^{{\pi^*}}_{t}\mathrm{d}t +\sigma_t  \alpha_tX^{{\pi^*}}_{t}\mathrm{d}W_{\cal H}(t) \\
& \hspace{4em}+  \tilde\psi^{{\pi^*}}_t\Bigg[r_t+(\mu(t,\pi_t^*)-r_t)\pi_t^*  +\sigma_t\pi_t^*\phi_t  \Bigg]\mathrm{d}t+\tilde\psi^{{\pi^*}}_t \sigma_t\pi^*_t \mathrm{d}W_{\cal H}(t), \quad 0\le t\le T,
 \\
  &\tilde\psi^{{\pi^*}}_0=0  ,\end{aligned} \right.   
 \end{eqnarray*}
  \begin{eqnarray*}
  \left\{ \begin{aligned}&\mathrm{d}\psi^{{\theta^*}}_t=\varepsilon_{t}^{{\theta^*}}\beta_t\mathrm{d}W_{\cal H}(t)+\psi^{{\theta^*}}_t     \theta^*_t \mathrm{d}W_{\cal H}(t) , \quad 0\le t\le T,
  \\&\psi^{{\theta^*}}_0=0 .\end{aligned} \right.   
 \end{eqnarray*}
 \end{assumption}
 
  \begin{assumption}\label{assump4}
  Let $ \beta_s= \xi 1_{(t,t+h]}(s) $, $0\le s\le T$, for fixed $0\le t<t+h\le T$, where the random variable $\xi$ is of the form $1_{A_t}$ for any ${\cal H}_t$-measurable set $A_t$. Then $\beta\in {\cal A}_2'$.
  \end{assumption}
  
Now we define the Hamiltonian $H:[0,T]\times\mathbb{R}\times\mathbb{R}\times\mathbb{R}  \times \mathbb{R}   \times\mathbb{R}^2\times\mathbb{R}^2\times\Omega\rightarrow\mathbb{R} $ by 
  \begin{equation*}
 \begin{aligned}
H(t,x,\varepsilon,\pi,\theta,p,q,\omega):=&g(\theta)\varepsilon+\left[r_t+(\mu(t,\pi)-r_t)\pi +\sigma_t\pi \phi_t \right]xp_1 +\sigma_t\pi xq_1 +
\varepsilon\theta q_{2},
 \end{aligned}
 \end{equation*}
 where $p=\begin{pmatrix}p_{1}\\p_{2}\end{pmatrix}$, and $q=\begin{pmatrix}q_{1}\\q_{2}\end{pmatrix}$. It is obvious that $H$ is differentiable with respect to $x$, $\varepsilon$, $\pi$ and $\theta$. The associated BSDE system for the adjoint pair $(p_t,q_t)$ is given by  
  \begin{eqnarray}\label{adjoint0}
  \left\{ \begin{aligned}&\mathrm{d}p_1(t)=-\frac{\partial H}{\partial x}(t)\mathrm{d}t+q_{1}(t)\mathrm{d}W_{\cal H}(t), \quad 0\le t\le T,
  \\&p_1(T)=\varepsilon^{\theta}_T (X^{\pi}_T) ^{-1},\end{aligned} \right.   
 \end{eqnarray}
 and
 \begin{eqnarray}\label{adjoint}
  \left\{ \begin{aligned}&\mathrm{d}p_2(t)=-\frac{\partial H}{\partial \varepsilon}(t)\mathrm{d}t+q_{2}(t)\mathrm{d}W_{\cal H}(t), \quad 0\le t\le T,
  \\&p_2(T)=\ln X^{\pi}_T,\end{aligned} \right.   
 \end{eqnarray}
 where $p_{i}(t)$ is a continuous ${\cal H}_t$-semimartingale, and $q_{i}(t)$ is an ${\cal H}_t$-adapted process with the following integrability
  \begin{eqnarray*}
\int_0^T  \left[\big|\frac{\partial H}{\partial x}(t)\big|+\big|\frac{\partial H}{\partial \varepsilon}(t)\big|+|q_{i}(t)|^2 \right]\mathrm{d}t <\infty,
  \end{eqnarray*}
$ i=1,2$. Here and in the following the abbreviated notation $  H(t):= H(t,X_t^\pi,\varepsilon^\theta_t,\pi_t,\theta_t,p_t,q_t,\omega) $, etc., are taken.

We give a necessary maximum principle to characterize the optimal pair $({\pi^*},{\theta^*})\in {\cal A}_1'\times{\cal A}_2'$.

   \begin{theorem}\label{mainth4}
Suppose $({\pi^*},{\theta^*})\in {\cal A}_1'\times{\cal A}_2'$ is optimal for Problem \ref{sdg} under Assumptions \ref{assump1}, \ref{assump2}, \ref{assump3},  \ref{assump3_2} and \ref{assump4}, and $(p^*,q^*)$ is the associated adjoint pair satisfying BSDEs (\ref{adjoint0}) and (\ref{adjoint}). Then $({\pi^*},{\theta^*})$ solves the following equations (the Hamiltonian system)
   \begin{equation} \label{half_equ3_0}
 \begin{aligned}
 \frac{\partial H^*}{\partial \pi}(t) =0,\quad 0\le t\le T,
\end{aligned}
 \end{equation}
and
  \begin{equation} \label{half_equ3}
 \begin{aligned}
 \frac{\partial H^*}{\partial \theta}(t) =0,\quad 0\le t\le T,
\end{aligned}
 \end{equation}
 given the following integrability conditions 
   \begin{equation*} 
 \begin{aligned}
\mathbb{E}\Bigg\{&\int_0^T (\tilde\psi^{{\pi^*}}_t)^2 \mathrm{d}\big\langle p^*_{1}\big\rangle_t +\int_0^T p^*_1(t)^2\mathrm{d}\big\langle\tilde\psi^{{\pi^*}}\big\rangle_t      \Bigg\} <\infty, 
\end{aligned}
 \end{equation*}
 and
  \begin{equation*} 
 \begin{aligned}
\mathbb{E}\Bigg\{&\int_0^T (\psi^{{\theta^*}}_t)^2\mathrm{d}\big\langle p^*_2\big\rangle_t +\int_0^T p^*_2(t)^2\mathrm{d}\big\langle\psi^{{\theta^*}}\big\rangle_t  \Bigg\} <\infty, 
\end{aligned}
 \end{equation*}
  for all bounded $(\alpha,\beta)\in{\cal A}_1'\times{\cal A}_2'$. Here, $H^*(t):=H(t,X^{{\pi^*}}_t,\varepsilon^{{\theta^*}}_t,\pi^*_t,\theta^*_t,p^*_t,q^*_t,\omega)$, etc.
\end{theorem}
\begin{proof}
 Suppose that the pair $({\pi^*},{\theta^*})\in {\cal A}_1'\times{\cal A}_2'$ is optimal. Then for any bounded $\beta\in {\cal A}_2' $ and $|y|<\delta$, we have $J({\pi^*},{\theta^*}+y\beta)\ge J({\pi^*},{\theta^*}) $, which implies that $y=0$ is a minimum point of the function $y\mapsto J({\pi^*},{\theta^*}+y\beta)$. By Assumptions \ref{assump3} and \ref{assump3_2}, and It\^{o} formula for It\^{o} integrals, we have 
  \begin{equation} 
 \begin{aligned}
\frac{\mathrm{d}}{\mathrm{d}y}J({\pi^*},{\theta^*}+y\beta){|}_{y=0}&=\mathbb{E}\Bigg[ \psi^{{\theta^*}}_T \ln X_T^{{\pi^*}}+\int_0^T \psi^{{\theta^*}}_s g(\theta^*_s)\mathrm{d}s+\int_0^T\varepsilon^{{\theta^*}}_s g'(\theta^*_s)\beta_s\mathrm{d}s
\Bigg]\\
&=\mathbb{E}\Bigg[ \psi^{{\theta^*}}_Tp^*_{2}(T)+\int_0^T \psi^{\theta^*}_sg(\theta^*_s)\mathrm{d}s+\int_0^T\varepsilon^{{\theta^*}}_s g'(\theta^*_s)\beta_s\mathrm{d}s
\Bigg]\\
&=\mathbb{E}\Bigg[ \int_0^T\psi^{{\theta^*}}_s\mathrm{d}p^*_2(s)+\int_0^Tp^*_2(s)\mathrm{d}\psi^{{\theta^*}}_s+\big\langle p^*_2,\psi^{{\theta^*}}\big\rangle_T\\
&\hspace{2em}\  +\int_0^T \psi^{{\theta^*}}_s g(\theta^*_s)\mathrm{d}s+\int_0^T\varepsilon^{{\theta^*}}_s g'(\theta^*_s)\beta_s\mathrm{d}s
\Bigg]
 \\
&=\mathbb{E}\Bigg[  \int_0^Tq^*_{2}(s)\varepsilon^{{\theta^*}}_{s}\beta_s\mathrm{d}s+\int_0^T\varepsilon^{{\theta^*}}_s g'(\theta^*_s)\beta_s\mathrm{d}s\Bigg]\\
&=\mathbb{E}\Bigg[  \int_0^T \frac{\partial H^*}{\partial \theta}(s) \beta_s\mathrm{d}s\Bigg]\\
&=0.
 \end{aligned}
 \end{equation}
By Assumption \ref{assump4} and the same procedure in Theorem \ref{mainth1}, we can deduce that $ \frac{\partial H^*}{\partial \theta}(t) =0$, $t\in[0,T]$. By similar arguments, we can conclude that $ \frac{\partial H^*}{\partial \pi}(t) =0$, $t\in[0,T]$.
\end{proof}

Combining Theorem \ref{mainth4} with the conclusion in Section \ref{half_section}, we give the total characterization of the optimal pair $({\pi^*},{\theta^*})$ as the following theorem.

   \begin{theorem}\label{mainth5}
Suppose $({\pi^*},{\theta^*})\in {\cal A}_1'\times{\cal A}_2'$ is optimal for Problem \ref{sdg} (with the associated pair $(p^*,q^*)$  satisfying BSDEs (\ref{adjoint0}) and (\ref{adjoint})) under the conditions of Theorem \ref{mainth4}. Then $({\pi^*},{\theta^*})$ solves equations (\ref{halfequ_21}), (\ref{half_equ3_0}) and (\ref{half_equ3}). 
 \end{theorem}

\begin{remark}
In fact, equations (\ref{halfequ_21}) and (\ref{half_equ3}) are enough to obtain the optimal pair $({\pi^*},{\theta^*})$. This combined method (rather than only using the Hamiltonian system  (\ref{half_equ3_0})-(\ref{half_equ3})) can always give a better characterization of $({\pi^*},{\theta^*})$ since the relationship of $\pi^*$ and $\theta^*$ is given by (\ref{halfequ_21}). Moreover, when the mean rate of return $\mu$ is dependent on $\pi^*$, i.e., a large insurer is consider, it is very hard to obtain the solution $({\pi^*},{\theta^*})$ by just using (\ref{half_equ3_0})-(\ref{half_equ3}) since the dynamic of $X^\pi$ is not homogeneous in this situation (see \cite{Oksendal19}), while it is not by combining the Hamiltonian system with (\ref{halfequ_21}). We will give examples in the next section to illustrate this.
\end{remark}

\section{The small insider case: maximum principle}
\label{sec_exampleg_sub1}
We assume that the mean rate of return function $\mu(t,x)=\mu_0(t)+\varrho_tx$ for some ${\cal F}_t$-adapted measurable processes $\mu_0(t)$ and $\varrho_t$ with $0\le \varrho_t< \frac{1}{2}\sigma_t^2$. Put $\iota_t=\frac{\mu_0(t)-r_t}{\sigma_t}$, $\tilde\sigma_t=\sigma_t-\frac{2\varrho_t}{\sigma_t}$, and $\tilde \phi_t=\iota_t+\phi_t$. Assume further the penalty function $g$ is of the quadratic form, i.e., $g(\theta)=\frac{1}{2}\theta^2 $. Then we have by the Girsanov theorem that
     \begin{equation*} 
 \begin{aligned}
\mathbb{E}\left[  \int_0^T\varepsilon_s^{\theta^*}  g(\theta^*_s)\mathrm{d}s \right]
= \mathbb{E}_{{\cal Q}^{\theta^*}}\left[\int_0^Tg(\theta^*_s)\mathrm{d}s   \right]
&=\mathbb{E}_{{\cal Q}^{\theta^*}}\left[  \int_0^T\theta^*_s\mathrm{d}W_{\cal H}(s)  -\ln \varepsilon^{\theta^*}_T  \right]\\
&=\mathbb{E}_{{\cal Q}^{\theta^*}}\left[  \int_0^T(\theta^*_s)^2 \mathrm{d}s -\ln \varepsilon^{\theta^*}_T  \right]\\
&=2\mathbb{E}_{{\cal Q}^{\theta^*}}\left[\int_0^Tg(\theta^*_s)\mathrm{d}s   \right]-\mathbb{E}_{{\cal Q}^{\theta^*}}\left[  \ln \varepsilon^{\theta^*}_T  \right],
\end{aligned}
 \end{equation*}
 which implies that 
  \begin{equation}\label{part1ofJ}
  \mathbb{E}\left[  \int_0^T\varepsilon_s^{\theta^*}  g(\theta^*_s)\mathrm{d}s \right]
= \mathbb{E}_{{\cal Q}^{\theta^*}}\left[  \ln \varepsilon^{\theta^*}_T  \right]= \mathbb{E}\left[\varepsilon^{\theta^*}_T  \ln \varepsilon^{\theta^*}_T  \right].
\end{equation}  

We make the following assumption before our procedure.

\begin{assumption}\label{extraa1}
Suppose the following integrability condition holds
\begin{equation}
\int_0^T |\tilde\phi_t|^2\mathrm{d}t<\infty.
\end{equation}
\end{assumption}

By the Hamiltonian system (\ref{half_equ3}) in Theorem \ref{mainth4}, we have
 \begin{eqnarray} \label{logconti_v}
\frac{\partial H^*}{\partial \theta}(t)=  \theta^*_t+q_{2}^*(t)  =0,
 \end{eqnarray}
 Substituting (\ref{logconti_v}) into the adjoint BSDE (\ref{adjoint}) with respect to $p_2^*(t)$ we have
 \begin{eqnarray}\label{equeg1_39}
  \left\{ \begin{aligned}& \mathrm{d}p_2^*(t)=\frac{(\theta_t^*)^2}{2} \mathrm{d}t -\theta^*_t\mathrm{d}W_{\cal H}(t) , \quad 0\le t\le T,
  \\&p_2^*(T)=\ln X^{\pi^*}_T.\end{aligned} \right.   
 \end{eqnarray}
The SDE (\ref{half_equ2}) of $\varepsilon^{\theta^*}_t$ implies that
    \begin{equation} \label{equeg1_40}
  \begin{aligned}
  \mathrm{d} \ln\varepsilon^{\theta^*}_t =-\frac{(\theta_t^*)^2}{2}\mathrm{d}t+  \theta^*_t\mathrm{d}W_{\cal H}(t).
   \end{aligned}
 \end{equation}
 By comparing (\ref{equeg1_39}) with (\ref{equeg1_40}), the solution of the BSDE (\ref{equeg1_39}) can be expressed as
     \begin{equation}  \label{equeg1_41}
  \begin{aligned}
 p_2^*(t)=p^*_2(0)-\ln \varepsilon^{\theta^*}_t.
 \end{aligned}
 \end{equation}
Denote the ${\cal H}_0$-measurable random variable $p^*_2(0)$ by $c^*_2$. Substituting the terminal condition in (\ref{equeg1_39}), i.e., $p_2^*(T)=\ln X^{\pi^*}_T$, into (\ref{equeg1_41}) with $t=T$, we have
     \begin{equation} \label{equeg1_42fb}
  \begin{aligned}
\ln\big( \varepsilon^{\theta^*}_TX_T^{\pi^*}\big)=c^*_2.
 \end{aligned}
 \end{equation} 

Since $\varepsilon^{\theta^*}_t$ is an ${\cal H}_t$-martingale, we have $\varepsilon^{\theta^*}_t=\mathbb{E}\left[ \varepsilon^{\theta^*}_T  |{\cal H}_t\right]=\mathbb{E}\left[ e^{c^*_2} (X^{\pi^*}_T)^{-1} \big{|}{\cal H}_t\right]$ by (\ref{equeg1_42fb}). Using that $\varepsilon^{\theta^*}_0=1$ we obtain
 \begin{equation}  \label{anequaboutc2}
  \begin{aligned}
  e^{c^*_2}=\frac{1}{\mathbb{E}\left[  (X^{\pi^*}_T)^{-1} |{\cal H}_0\right]}.
   \end{aligned}
 \end{equation} 
Thus, we can give the expression of $\varepsilon^{\theta^*}_t$ as follows
 \begin{equation}  \label{anequaboutc22}
  \begin{aligned}
  \varepsilon^{\theta^*}_t=\mathbb{E}\left[ \left(\mathbb{E}\left[  (X^{\pi^*}_T)^{-1} |{\cal H}_0\right] X^{\pi^*}_T  \right)^{-1} \big{|}{\cal H}_t\right].
   \end{aligned}
 \end{equation} 
Moreover, substituting (\ref{anequaboutc2}) into (\ref{equeg1_42fb}) we obtain
\begin{equation}  \label{anequaboutc222}
  \begin{aligned}
  \varepsilon^{\theta^*}_T= \frac{1}{  \mathbb{E}\left[  (X^{\pi^*}_T)^{-1} |{\cal H}_0\right] X^{\pi^*}_T  }.
   \end{aligned}
 \end{equation} 

On the other hand, by the Hamiltonian system (\ref{half_equ3_0}) in Theorem \ref{mainth4}, we have
 \begin{equation}\begin{aligned} 
\frac{\partial H^*}{\partial \pi}(t)=   \left(\frac{\partial \mu}{\partial x}(t,\pi^*_t)+(\mu(t,\pi^*_t)-r_t)+\sigma_t\phi^*_t  \right )X^{\pi^*}_t p_1^*(t) + \sigma_t X^{\pi^*}_t q^*_{1}(t)  =0,
\end{aligned}  \end{equation} 
which implies that
  \begin{equation}\label{generalconti_v}
  \begin{aligned}
    \left(\frac{\partial \mu}{\partial x}(t,\pi^*_t)+(\mu(t,\pi^*_t)-r_t)+\sigma_t\phi^*_t  \right )p_1^*(t) + \sigma_tq^*_{1}(t) =0 .
  \end{aligned}
 \end{equation}
 Substituting (\ref{generalconti_v}) into the adjoint BSDE (\ref{adjoint0}) with respect to $p_1^*(t)$ yields
 \begin{eqnarray}\label{idontknow0}
   \left\{ \begin{aligned}&   \mathrm{d}p_1^*(t)=-\left[r_t-\frac{(\sigma_t-\tilde\sigma_t)\sigma_t\pi^*_t}{2} \right]p_1^*(t)\mathrm{d}t-\left[\frac{\sigma_t-\tilde\sigma_t}{2}(1+\pi^*_t)+\tilde\phi_t\right]p_1^*(t)\mathrm{d}W_{\cal H}(t), \quad 0\le t\le T,
  \\&  p_1^*(T)=\varepsilon^{\theta^*}_T (X_T^{\pi^*})^{-1}.\end{aligned} \right.  
 \end{eqnarray}
Suppose $\varrho_t\equiv 0$, i.e., $\sigma_t-\tilde\sigma_t\equiv0$, then (\ref{idontknow0}) degenerates to
  \begin{eqnarray}\label{idontknow}
   \left\{ \begin{aligned}&   \mathrm{d}p_1^*(t)=- r_t p_1^*(t)\mathrm{d}t- \tilde\phi_t p_1^*(t)\mathrm{d}W_{\cal H}(t) , \quad 0\le t\le T,
  \\&  p_1^*(T)=\varepsilon^{\theta^*}_T(X_T^{\pi^*})^{-1},\end{aligned} \right.  
 \end{eqnarray}
which implies that all coefficients in (\ref{idontknow}) are independent of $\pi^*$. Then the unique solution of (\ref{idontknow}) is given by
   \begin{equation} \label{ireallydontknow}
  \begin{aligned}
 p_1^*(t)=c^*_1 \Pi^*(0,t),
  \end{aligned}
 \end{equation}
 where $c_1^*:=p^*_1(0)$ is an ${\cal H}_0$-measurable random variable, and $\Pi^*(t_1,t_2)$, $0\le t_1\le t_2\le T$, is defined as
    \begin{equation} \label{definitionofPi}
  \begin{aligned}
\Pi^*(t_1,t_2):= \exp\Bigg\{ & -\int_{t_1}^{t_2} r_s  \mathrm{d}s
  -\int_{t_1}^{t_2} \tilde\phi_s  \mathrm{d}W_{\cal H}(s)
    -\frac{1}{2}\int_{t_1}^{t_2}    \tilde\phi_s^2\mathrm{d}s \Bigg\}.
  \end{aligned}
 \end{equation}
Substituting (\ref{ireallydontknow}) into (\ref{idontknow}) with $t=T$ we have
    \begin{equation} \label{dffunctionofu}
  \begin{aligned}
 X^{\pi^*}_T= \frac{\varepsilon^{\theta^*}_T}{c_1^*\Pi^*(0,T)}.
  \end{aligned}
 \end{equation}
 Combining (\ref{dffunctionofu}) with (\ref{anequaboutc222}) we have
  \begin{equation} \label{dffunctionofu222}
  \begin{aligned}
 X^{\pi^*}_T=  \frac{1}{\sqrt{ c_3^*\Pi^*(0,T)}},
  \end{aligned}
 \end{equation}
 where $c_3^*:=c^*_1\mathbb{E}\left[ (X^{\pi^*}_T)^{-1}  |{\cal H}_0\right]$ is also an ${\cal H}_0$-measurable random variable.

Put $  z^*_t=   \sigma_t\pi^*_tX^{\pi^*}_t$. Then we have
  \begin{equation}\label{bsde_barzt1}
   \begin{aligned}
 &\pi_t^*=\frac{  z^*_t}{\sigma_tX^{\pi^*}_t}.   
\end{aligned} 
 \end{equation}
The SDE (\ref{wealthsde2}) of $X_t^{\pi^*}$ combined with (\ref{bsde_barzt1}) leads to the following linear BSDE
  \begin{eqnarray}\label{final_gex_1}
  \left\{ \begin{aligned}&\mathrm{d}X^{\pi^*}_t=-f_{\text{L}}(t,X^{\pi^*}_t,z^*_t,\omega)\mathrm{d}t+z^*_t\mathrm{d}W_{\cal H}(t), \quad 0\le t\le T,
  \\& X^{\pi^*}_T=\frac{1}{\sqrt{ c_3^*\Pi^*(0,T)}},\end{aligned} \right.   
 \end{eqnarray}
where the generator (or the driver) $f_{\text{L}}:[0,T]\times \mathbb{R}\times \mathbb{R}\times\Omega\rightarrow \mathbb{R}$  is given by
 \begin{equation*} 
\begin{aligned}
f_{\text{L}}(t,x,z,\omega)=-r_tx-\tilde\phi_tz.
\end{aligned} 
 \end{equation*}
 
 \begin{remark} \label{explain_log_half}
If $\varrho_t\neq 0$, the terminal condition in BSDE (\ref{final_gex_1}) will depend on $z^*_t$  by (\ref{idontknow0}), which makes the BSDE (\ref{final_gex_1}) irregular and very hard to solve. The reason is that the SDE (\ref{wealthsde2}) of $X^{\pi^*}$ is not homogeneous in this situation. However, once the relationship (equations (\ref{halfequ_21}) in Theorem \ref{mainth6}) of $\pi^* $ and $\theta^*$ is solved, we can then overcome this situation by a combined method in Section \ref{sec_example1}.
 \end{remark}
 
By the It\^{o} formula for It\^{o} integrals, we have
 \begin{equation}\label{solutionoflbsde}
 \begin{aligned}
\mathrm{d}\left(\Pi^*(0,t)X^{\pi^*}_t\right)&=\Pi^*(0,t)\mathrm{d}X^{\pi^*}_t+X^{\pi^*}_{t}\mathrm{d}\Pi^*(0,t)+\mathrm{d}\big\langle X^{\pi^*},\Pi^*(0,\cdot)\big\rangle_t\\
&=\Pi^*(0,t) \left( z^*_t-\tilde\phi_tX^{\pi^*}_t  \right)\mathrm{d}W_{\cal H}(t).
\end{aligned}
\end{equation}
Suppose the following integrability condition holds
 \begin{equation}\label{linearbsde_bdg}
 \begin{aligned}
\mathbb{E}\Bigg[\left(\int_0^T  \Pi^*(0,t)^2 \left(z^*_t-\tilde\phi_tX^{\pi^*}_t\right)^2\mathrm{d}t \right)^{\frac{1}{2}}+ (X^{\pi^*}_T)^2 \Bigg]<\infty.
\end{aligned}
\end{equation}
Then by the Burkholder-Davis-Gundy inequality (see \cite{Karatzas91}), $\Pi^*(0,t)X^{\pi^*}_t$, $t\in[0,T]$, is an ${\cal H}_t$-martingale. By taking the expectation in (\ref{solutionoflbsde}) we have
 \begin{equation}\label{linearbsde_bdg_2}
 \begin{aligned}
 X^{\pi^*}_t= \frac{1}{\sqrt{c_3^*}}\mathbb{E} \left[ \frac{   \Pi^*(t,T) }{\sqrt{  \Pi^*(0,T)}} | {\cal H}_t\right].
\end{aligned}
\end{equation}
Substituting (\ref{linearbsde_bdg_2}) into the initial value condition $X^{\pi^*}_0=X_0$ with $t=0$ yields
 \begin{equation} \label{c3determined}
 \begin{aligned}
\frac{1}{\sqrt{c^*_3}}=\frac{X_0}{\mathbb{E} \left[  \sqrt{  \Pi^*(0,T)} | {\cal H}_0\right]}.
\end{aligned}
\end{equation}
Substituting (\ref{c3determined}) into (\ref{final_gex_1}) yields
  \begin{eqnarray}\label{final_gex_1_i}
  \left\{ \begin{aligned}&\mathrm{d}X^{\pi^*}_t=-f_{\text{L}}(t,X^{\pi^*}_t,z^*_t,\omega)\mathrm{d}t+z^*_t\mathrm{d}W_{\cal H}(t), \quad 0\le t\le T,
  \\& X^{\pi^*}_T=\frac{X_0}{\mathbb{E} \left[  \sqrt{  \Pi^*(0,T)} | {\cal H}_0\right]\sqrt{\Pi^*(0,T)}},\end{aligned} \right.   
 \end{eqnarray}
where the generator $f_{\text{L}}:[0,T]\times \mathbb{R}\times \mathbb{R}\times\Omega\rightarrow \mathbb{R}$  is given by
 \begin{equation*} 
\begin{aligned}
f_{\text{L}}(t,x,z,\omega)=-r_tx-\tilde\phi_tz.
\end{aligned} 
 \end{equation*}
Combining (\ref{linearbsde_bdg_2}) with (\ref{c3determined}) we obatin
 \begin{equation}\label{linearbsde_bdg_2_ii}
 \begin{aligned}
 X^{\pi^*}_t=\frac{X_0\mathbb{E} \left[   \sqrt{  \Pi^*(t,T)}  | {\cal H}_t\right]}{\mathbb{E} \left[  \sqrt{  \Pi^*(0,T)} | {\cal H}_0\right]\sqrt{  \Pi^*(0,t)}},
\end{aligned}\end{equation}
and
 \begin{equation}\label{linearbsde_bdg_2_iiT}
 \begin{aligned}
 X^{\pi^*}_T=\frac{X_0 }{\mathbb{E} \left[  \sqrt{  \Pi^*(0,T)} | {\cal H}_0\right]\sqrt{  \Pi^*(0,T)}}.
\end{aligned}
\end{equation}

On the other hand, substituting (\ref{linearbsde_bdg_2_ii}) into (\ref{anequaboutc22}) and (\ref{anequaboutc222}) we obtain
 \begin{equation}  \label{anequaboutc22_final}
  \begin{aligned}
  \varepsilon^{\theta^*}_t= \frac{\mathbb{E}\left[     { \sqrt{\Pi^*(0,T)}} \big{|}{\cal H}_t\right]} { \mathbb{E} \left[  \sqrt{  \Pi^*(0,T)} \big| {\cal H}_0\right]  },
   \end{aligned}
 \end{equation} 
and
\begin{equation}  \label{anequaboutc222_final}
  \begin{aligned}
  \varepsilon^{\theta^*}_T=\frac{  { \sqrt{\Pi^*(0,T)}} } { \mathbb{E} \left[  \sqrt{  \Pi^*(0,T)} \big| {\cal H}_0\right]  }.
   \end{aligned}
 \end{equation} 
Moreover, $\theta^*$ can be obtained by the equation (\ref{halfequ_21}) in Theorem \ref{mainth6} as follows
 \begin{equation}\label{reduceth6ii} 
 \begin{aligned}
 \theta_t^*=  \sigma_t\pi^*_t-\tilde\phi_t= \frac{  z^*_t}{X^{\pi^*}_t}-\tilde\phi_t. 
  \end{aligned}
 \end{equation}

By (\ref{part1ofJ}), (\ref{equeg1_42fb}), (\ref{anequaboutc2}) and (\ref{final_gex_1_i}), we can calculate the value of Problem \ref{sdg} as follows
\begin{equation}  \label{dontknowv}
  \begin{aligned}
 V&=\mathbb{E}\left[\varepsilon_T^{\theta^*}\ln \left( \varepsilon^{\theta^*}_TX^{\pi^*}_T  \right)    \right]  
 =\mathbb{E}\left[\varepsilon_T^{\theta^*} c^*_2 \right] =\mathbb{E}\left[-\varepsilon_T^{\theta^*}\ln \mathbb{E} [  (X_T^{\pi^*})^{-1}|{\cal H}_0  ] \right]\\
 &=-\mathbb{E}\left[ \mathbb{E}\left[\varepsilon_T^{\theta^*}\ln \mathbb{E} [  (X_T^{\pi^*})^{-1}|{\cal H}_0  ]\big{|} {\cal H}_0 \right]  \right]=-\mathbb{E}\left[ \ln \mathbb{E} [  (X_T^{\pi^*})^{-1}|{\cal H}_0  ]\mathbb{E}\left[\varepsilon_T^{\theta^*} \big{|} {\cal H}_0 \right]  \right]\\
 &=-\mathbb{E}\left[ \ln \mathbb{E} [  (X_T^{\pi^*})^{-1}|{\cal H}_0  ]  \right]=\ln X_0-2\mathbb{E}\left[ \ln \mathbb{E}( \sqrt{\Pi^*(0,T)}  |{\cal H}_0)  \right].
   \end{aligned}
 \end{equation} 
 
Further, if the filtration $\{{\cal H}_t\}_{0\le t\le T}$ is the augmentation of the natural filtration of $W_{\cal H}(t)$, which was also assumed in \cite{Biagini05}, then ${\cal H}_0$ is generated by the trivial $\sigma$-algebra and all $\mathbb{P}$-negligible sets. Suppose that the right hand of (\ref{linearbsde_bdg_2_iiT}) is $L^2$-integrable, and $r$ and $\tilde \phi$ are bounded. Then by \cite[Theorem 4.8]{Oksendal19}, the linear BSDE (\ref{final_gex_1_i}) has a unique strong solution $(X^{\pi^*},z^*)$, i.e., $X^{\pi^*}_t$ is a continuous ${\cal H}_t$-adapted process with $\mathbb{E}\left[\sup_{0\le t\le T}|X^{\pi^*}_t|^2 \right]<\infty$, $z^*_t$ is a measurable ${\cal H}_t$-adapted process with $\mathbb{E}\left[ \int_0^T|z^*_t|^2\mathrm{d}t \right]<\infty$, and $(X^{\pi^*},z^*)$ satisfies the BSDE (\ref{final_gex_1}). 

\begin{remark}
Under mild conditions, we can obtain the formulae for $  z^*_t $ as follows (see \cite[Proposition 3.5.1]{Delong13})
 \begin{eqnarray}\label{final_ex_bmoofz1}
  \begin{aligned}& z^*_t=D_t X_t^{\pi^*},
  \end{aligned}  
 \end{eqnarray}
 where $D_t$ is the Malliavin gradient operator from the Sobolev space $D^{1,2}(\Omega)$ to $L^2(\Omega\times[0,T])$. 
\end{remark}

To sum up, we give the following theorem.
 
    \begin{theorem} \label{ex_mainth_idontknow}
Assume that $\mu(t,x)=\mu_0(t) $ for some ${\cal F}_t$-adapted measurable process $\mu_0(t)$, and $g(\theta)=\frac{1}{2} \theta^2$. Suppose $(\pi^*,\theta^*)\in {\cal A}_1'\times{\cal A}_2'$ is optimal for Problem \ref{sdg} under the conditions of Theorem \ref{mainth4}. Suppose Assumption \ref{extraa1} and the integrability condition (\ref{linearbsde_bdg}) hold. Then $\pi^*$, $\theta^*$, $X^{\pi^*}_t$ and $\varepsilon^{\theta^*}_t$ are given by (\ref{bsde_barzt1}), (\ref{reduceth6ii}) , (\ref{linearbsde_bdg_2_ii}) and (\ref{anequaboutc22_final}), respectively, where $\Pi^*$ is given by (\ref{definitionofPi}), and $(X^{\pi^*},  z^*)$ solves the linear BSDE (\ref{final_gex_1_i}).The value $V$ is given by (\ref{dontknowv}). Furthermore, if $\{{\cal H}_t\}_{0\le t\le T}$ is the augmentation of the natural filtration of $W_{\cal H}(t)$, the right hand of (\ref{linearbsde_bdg_2_iiT}) is $L^2$-integrable, and $r$ and $\tilde \phi$ are bounded, then the linear BSDE (\ref{final_gex_1_i}) has a unique strong solution, and $z^*$ is given by (\ref{final_ex_bmoofz1}) under mild conditions.
 \end{theorem}

\begin{remark}
If the filtration $\{{\cal H}_t\}_{0\le t\le T}$ in Theorem \ref{ex_mainth_idontknow} is not  the augmentation of the natural filtration of $W_{\cal H}(t)$, or the coefficients of the generator $f_{\text{L}}$ is not necessarily bounded, we refer to \cite{Eyraoud05,Li06,Lu13,Wang07} for further results. In those cases, the existence and uniqueness of the solution to the BSDE (\ref{final_gex_1_i}) still hold under mild conditions when a general martingale representation property was assumed, or a transposition solution was considered, or a stochastic Lipschitzs condition was considered. 
\end{remark}

When the investor has no insider information, i.e., ${\cal H}_t={\cal F}_t$, we have $\phi=0$. Assume further that  the utility function is of the logarithmic form, i.e., $U(x)=\ln x $,  and all the parameter processes are assumed to be deterministic bounded functions. Then by (\ref{final_ex_bmoofz1}), we have
   \begin{equation}  
 \begin{aligned}
z_t^*&=D_tX^{\pi^*}_t\\
&=\frac{X_0\mathbb{E}\left[ D_t\sqrt{\Pi^*(t,T)  } |{\cal F}_t   \right]}{  \mathbb{E}\sqrt{\Pi^*(0,T) }\sqrt{\Pi^*(0,t)}}-\frac{1}{2} \frac{X_0\mathbb{E}\left[ \sqrt{\Pi^*(t,T)  } |{\cal F}_t   \right]}{  \mathbb{E}\sqrt{\Pi^*(0,T) }\Pi^*(0,t)^{\frac{3}{2}}}D^1_t\Pi^*(0,t)\\
&= \frac{1}{2} \frac{X_0\mathbb{E}\left[ \sqrt{\Pi^*(t,T)   } |{\cal F}_t   \right]}{  \mathbb{E}\sqrt{\Pi^*(0,T) }\sqrt{\Pi^*(0,t)}}\iota_t.
\end{aligned}
\end{equation}

By (\ref{bsde_barzt1}) we can obtain the robust optimal strategy
  \begin{equation}  \label{conti_noinsider}
 \begin{aligned}
 \pi^*_t&= \frac{\mu_0(t)-r_t}{2\sigma^2_t}.
\end{aligned}
\end{equation}

By the equation (\ref{halfequ_21}) in Theorem \ref{mainth6}, we can also obtain
 \begin{equation}\label{bsde_barzt1_donsktheta_0}
   \begin{aligned}
 &\theta_t^*=-\frac{\mu_0(t)-r_t}{2\sigma_t}. 
\end{aligned} 
 \end{equation}
 
By (\ref{dontknowv}), we obtain the value of Problem \ref{sdg} as follows
\begin{equation}  \label{value1}
  \begin{aligned}
 V&= \ln X_0-\ln  \left(\mathbb{E}\sqrt{\Pi^*(0,T)} \right)^2 \\
 &=\ln X_0+\int_0^Tr_t\mathrm{d}t+\frac{1}{4}\int_0^T\left( \frac{\mu_0(t)-r_t}{\sigma_t}    \right)^2\mathrm{d}t.
   \end{aligned}
 \end{equation} 
 
To sum up, we have the following result.

    \begin{corollary} \label{corex_mainth_idontknow_donsker}
Assume that $\mu(t,x)=\mu_0(t) $ for some bounded function $\mu_0(t)$, and $g(\theta)=\frac{1}{2} \theta^2 $. Assume further that ${\cal H}_t={\cal F}_t$ and all parameter processes are deterministic bounded functions. Suppose $(\pi^*,\theta^*)\in {\cal A}_1'\times{\cal A}_2'$ is optimal for Problem \ref{sdg} under the conditions of Theorem \ref{mainth4}. Then $(\pi^*,\theta^*)$ is given by (\ref{conti_noinsider}) and (\ref{bsde_barzt1_donsktheta_0}), and the value $V$ is given by (\ref{value1}).
 \end{corollary}

\subsection{A particular case}\label{Sub_particular1}
Next, we give a particular case to derive the closed form of the robust optimal strategy. Assume that the filtration is of initial enlargement type, i.e.,  
\begin{equation}\label{donsker_filtration}
{\cal H}_t=\bigcap_{s>t}({\cal F}_s\vee Y_0):=\bigcap_{s>t}\left({\cal F}_s\vee \int_0^{T_0}\varphi_{u}\mathrm{d}W_{u}\right),\quad 0\le t\le T,
\end{equation}
for some $T_0>T$, and all the parameter processes are assumed to be deterministic bounded functions. Here, $\varphi_t $ is some deterministic function satisfying $\Vert \varphi\Vert^2_{[s,t]}:=\int_s^{t}\varphi_u^2    \mathrm{d}u<\infty$ for all $0\le s\le t\le T_0$, and $\Vert \varphi\Vert^2_{[T,T_0]}>0$. 

In this situation, each ${\cal H}_t$-adapted process $x_t$, $t\in[0,T]$, has the form $x_t=x_1(t,Y_0,\omega)$ for some function $x_1: [0,T]\times\mathbb{R}\times\Omega\rightarrow \mathbb{R}$ such that $x_1(t,y)$ is ${\cal F}_t$-adapted for every $y\in\mathbb{R}$.  For simplicity, we write $x$ instead of $x_1$ in the sequel. To solve the anticipating linear BSDE (\ref{final_gex_1}), we need to introduce some white noise techniques in Malliavin calculus (see \cite{Draouil15,Huang00,DiNunno09}).

\begin{definition}[Donsker $\delta$ functional]\label{definitionofdonsker}
Let $Y:\Omega\rightarrow \mathbb{R}$ be a random variable which belongs to the distribution space $({\cal S})^{-1}$ (see \cite{Huang00} for the definition). Then a  continuous linear operator $\delta_{\cdot}(Y):\mathbb{R}\rightarrow ({\cal S})^{-1}$ is called a Donsker $\delta$ functional of $Y$ if it has the property that
\begin{equation*}
 \int_\mathbb{R}f(y)\delta_y(Y)\mathrm{d}y=f(Y)
\end{equation*}
for all Borel measurable functions $f:\mathbb{R}\rightarrow\mathbb{R}$ such that the integral converges in $({\cal S})^{-1}$.
\end{definition}

The following lemma gives a sufficient condition for the existence of the Donsker $\delta$ functional. The proof can be found in \cite{DiNunno09}.

\begin{lemma} \label{lemmaofdonsker}
Let $Y: \Omega\rightarrow \mathbb{R}$ be a Gaussian random variable with mean $\bar \mu$ and variance $\bar\sigma^2>0$. Then its Donsker $\delta$ functional $\delta_{y}(Y)$ exists and is uniquely given by
\begin{equation*}
\delta_y(Y)=\frac{1}{\sqrt{2\pi\bar\sigma^2}}\exp^{\diamond}\left\{-\frac{(y-Y)^{\diamond 2}}{2\bar\sigma^2}  \right\}\in ({\cal S})'\subset ({\cal S})^{-1},
\end{equation*}
\end{lemma}
where $({\cal S})'$ is the Hida distribution space, and $\diamond$ denotes the Wick product. We refer to \cite{Huang00} for relevant definitions.

Moreover, we can obtain the explicit expression of $\phi$ as the following lemma (see \cite{Draouil16}).

\begin{lemma}[Enlargement of filtration] \label{lemmaofdonskersemi}
Suppose $Y$ is an ${\cal F}_{T_0}$-measurable random variable for some $T_0> T$ and belongs to $({\cal S})'$. The Donsker $\delta$ functional of $Y$ exists and satisfies $\mathbb{E}[ \delta_\cdot(Y) |{\cal F}_t]\in L^2(\mathrm{m}\times\mathbb{P})$ and $\mathbb{E}[ D_t\delta_\cdot(Y) |{\cal F}_t]\in L^2(\mathrm{m}\times\mathbb{P})$, where $D_t$ is the (extended) Hida-Malliavin derivative (see \cite{DiNunno09}). Assume further that ${\cal H}_t=\bigcap_{s>t} ({\cal F}_s\vee Y)$, which satisfies the usual condition, and $W$ is an ${\cal H}_t$-semimartingale with the decomposition (\ref{decompofW}). Then we have
\begin{equation*}
\phi_t=\frac{\mathbb{E}[D_t\delta_{y}(Y)|{\cal F}_t]|_{y=Y}}{\mathbb{E}[\delta_y(Y)|{\cal F}_t]|_{y=Y}}.
\end{equation*}
\end{lemma}

By Lemma \ref{lemmaofdonsker} and the L\'{e}vy theorem, the Donsker $\delta$ functional of $Y_0$ in (\ref{donsker_filtration}) is given by
\begin{equation}
\delta_y(Y_0)=\frac{1}{\sqrt{2\pi\Vert \varphi\Vert_{[0,T_0]}^2}}\exp^{\diamond}\Bigg\{-\frac{(y-Y_0)^{\diamond 2}}{2\Vert\varphi \Vert_{[0,T_0]}^2}  \Bigg\},
\end{equation}
and we have
\begin{equation}\label{defofg}
\mathbb{E}[\delta_y(Y_0)|{\cal F}_t]=\frac{1}{\sqrt{2\pi\Vert \varphi\Vert_{[t,T_0]}^2}}\exp\Bigg\{-\frac{(y-\int_0^t\varphi_s\mathrm{d}W_s)^{ 2}}{2\Vert\varphi \Vert_{[t,T_0]}^2}  \Bigg\},\quad 0\le t\le T.
\end{equation}
 Moreover, when the filtration $\{{\cal H}_t\}_{0\le t\le T}$ is of the form (\ref{donsker_filtration}), we have by Lemma \ref{lemmaofdonskersemi} that
\begin{equation}\label{weknowphi1}
\phi_t=\phi_t(Y_0)=\frac{Y_{0}-\int_0^t\varphi_s\mathrm{d}W_s}{\Vert \varphi\Vert^2_{[t,T_0]}}\varphi_t.
\end{equation}
Substituting (\ref{weknowphi1}) into (\ref{definitionofPi}) and using the It\^{o} formula we can rewrite the expression of $\Pi^*(t_1,t_2)$, $0\le t_1\le t_2\le T$, as follows
    \begin{equation} \label{redefinitionofPi}
  \begin{aligned}
\Pi^*(t_1,t_2)&=\Pi^*(t_1,t_2,y)|_{y=Y_0}\\
&= \exp\left\{-\int_{t_1}^{t_2}\phi_s(Y_0)\mathrm{d}W_s+\frac{1}{2}\int_{t_1}^{t_2}\phi_s(Y_0)\mathrm{d}y\right\}\Pi^*_a(t_1,t_2)\\
&=\frac{\mathbb{E}[\delta_y(Y_0)|{\cal F}_{t_1}]}{\mathbb{E}[\delta_y(Y_0)|{\cal F}_{t_2}]}   \Pi^*_a(t_1,t_2) \big|_{y=Y_0},
  \end{aligned}
 \end{equation}
where
    \begin{equation}  \label{pi_adef}
  \begin{aligned}
\Pi^*_a(t_1,t_2):= \exp\Bigg\{  -\int_{t_1}^{t_2} r_s  \mathrm{d}s
  -\int_{t_1}^{t_2} \iota_s  \mathrm{d}W_s
   -\frac{1}{2}\int_{t_1}^{t_2}    \iota_s ^2 \mathrm{d}s \Bigg\}
  \end{aligned}
 \end{equation}
 is an ${\cal F}_t$-adapted semimartingale.
 
The terminal value condition of $X^{\pi^*}_T$ in (\ref{final_gex_1_i}) leads to
\begin{equation}\label{c3determined_donsker_iii}
\begin{aligned}
X^{\pi^*}_T=X^{\pi^*}_T(y)|_{y=Y_0}&=\frac{X_0}{\mathbb{E}\big[\sqrt{\Pi^*(0,T)} | {\cal H}_0\big](y)}\sqrt{\frac{ \mathbb{E}[\delta_y(Y_0)|{\cal F}_T]}{   \mathbb{E}[\delta_y(Y_0)|{\cal F}_0]   \Pi^*_a(0,T)}} \Bigg|_{y=Y_0}\\
&=\tilde c^*_3(y)  \sqrt{\frac{ \mathbb{E}[\delta_y(Y_0)|{\cal F}_T]}{    \Pi^*_a(0,T)}}\Bigg|_{y=Y_0},
\end{aligned}\end{equation}
where $\tilde c^*_3(y):=\frac{X_0}{\mathbb{E}\big[\sqrt{\Pi^*(0,T)} | {\cal H}_0\big](y) \sqrt{\mathbb{E}[\delta_y(Y_0)|{\cal F}_0]}  }$ is a Borel measurable function with respect to $y$.

By Definition \ref{definitionofdonsker}, the BSDE (\ref{final_gex_1_i}) can be rewritten as
  \begin{equation}\label{final_gex_1_donsker}\begin{aligned}
\int_{\mathbb{R}} X^{\pi^*}_t(y)\delta_y(Y_0)\mathrm{d}y=& \int_{\mathbb{R}}  X_T^{\pi^*}(y)\delta_y(Y_0)\mathrm{d}y+ \int_{\mathbb{R}} \int_t^T\left[-r_sX_s^{\pi^*}(y) -\iota_s z^*_s(y) \right] \mathrm{d}s \delta_y(Y_0) \mathrm{d}y\\
&-\int_{\mathbb{R}}\int_t^Tz^*_s(y)\mathrm{d}W_s\delta_y(Y_0) \mathrm{d}y, \quad 0\le t\le T.
\end{aligned} \end{equation}
 Then $(\ref{final_gex_1_i})$ holds if and only if we choose $(X^{\pi^*}_t(y),z^*_t(y))$ for each $y$ as the solution of the following classical linear BSDE with respect to the filtration $\{{\cal F}_t\}_{0\le t\le T}$
  \begin{eqnarray}\label{final_gex_1_donsker2}
  \left\{ \begin{aligned}&\mathrm{d}X^{\pi^*}_t(y)=-\bar f_{\text{L}}(t,X^{\pi^*}_t(y),z^*_t(y))\mathrm{d}t+z^*_t(y)\mathrm{d}W_t, \quad 0\le t\le T,
  \\& X^{\pi^*}_T(y)=\tilde c^*_3(y)   \sqrt{\frac{\mathbb{E}[\delta_y(Y_0)|{\cal F}_T]}{   \Pi^*_a(0,T)}} ,\end{aligned} \right.   
 \end{eqnarray}
 where the generator $\bar f_{\text{L}}:[0,T]\times \mathbb{R}\times \mathbb{R}^2 \rightarrow \mathbb{R}$  is given by
  \begin{equation*} 
\begin{aligned}
\bar f_{\text{L}}(t,x,z )=-r_tx-\iota_tz.
\end{aligned} 
 \end{equation*}
 
By \cite[Theorem 4.8]{Oksendal19}, the unique strong solution of (\ref{final_gex_1_donsker2}) is given by
  \begin{equation}\label{linearbsde_bdg_2_donsker}
 \begin{aligned}
 X^{\pi^*}_t(y)= \mathbb{E} \left[   \Pi_a^*(t,T)\tilde c^*_3(y)   \sqrt{\frac{ \mathbb{E}[\delta_y(Y_0)|{\cal F}_T]}{    \Pi^*_a(0,T)}} \bigg | {\cal F}_t\right],
\end{aligned}
\end{equation}
 By the initial value condition $X^{\pi^*}_0(y)=X_0$, the Borel measurable function $\tilde c^*_3(y)$ in (\ref{linearbsde_bdg_2_donsker}) is given by
 \begin{equation} \label{c3determined_donsker}
 \begin{aligned}
 \tilde c^*_3(y) =\frac{X_0}{\mathbb{E}     \sqrt{  \Pi_a^*(0,T) \mathbb{E}[\delta_y(Y_0)|{\cal F}_T]} }=\frac{X_0}{\mathbb{E}\big[\sqrt{\Pi^*(0,T)} | {\cal H}_0\big](y) \sqrt{\mathbb{E}[\delta_y(Y_0)|{\cal F}_0]}  }.
 \end{aligned}
\end{equation}
 The last equation in (\ref{c3determined_donsker}) is by the definition of $\tilde c^*_3(y)$. Substituting (\ref{c3determined_donsker}) into (\ref{linearbsde_bdg_2_donsker}) we obtain
   \begin{equation}\label{linearbsde_bdg_2_donsker_2}
 \begin{aligned}
 X^{\pi^*}_t(y)= \frac{X_0\mathbb{E}\left[ \sqrt{\Pi^*_a(t,T)\mathbb{E}[\delta_y(Y_0)|{\cal F}_T]  } |{\cal F}_t   \right]}{  \mathbb{E}\sqrt{\Pi^*_a(0,T)\mathbb{E}[\delta_y(Y_0)|{\cal F}_T]}\sqrt{\Pi^*_a(0,t)}}.
\end{aligned}
\end{equation}
By \cite[Proposition 3.5.1]{Delong13}, we have
   \begin{equation} \label{iknow1}
 \begin{aligned}
z_t^*(y)&=D_tX^{\pi^*}_t(y)\\
&=\frac{X_0\mathbb{E}\left[ D_t\sqrt{\Pi^*_a(t,T)\mathbb{E}[\delta_y(Y_0)|{\cal F}_T]  } |{\cal F}_t   \right]}{  \mathbb{E}\sqrt{\Pi^*_a(0,T)\mathbb{E}[\delta_y(Y_0)|{\cal F}_T]}\sqrt{\Pi^*_a(0,t)}}-\frac{1}{2} \frac{X_0\mathbb{E}\left[ \sqrt{\Pi^*_a(t,T)\mathbb{E}[\delta_y(Y_0)|{\cal F}_T]  } |{\cal F}_t   \right]}{  \mathbb{E}\sqrt{\Pi^*_a(0,T)\mathbb{E}[\delta_y(Y_0)|{\cal F}_T]}\Pi^*_a(0,t)^{\frac{3}{2}}}D_t\Pi^*_a(0,t)\\
&=\frac{1}{2}\frac{X_0\mathbb{E}\left[ \sqrt{ \Pi^*_a(t,T)\mathbb{E}[\delta_y(Y_0)|{\cal F}_T]   } \left(y-\int_0^T\varphi_s\mathrm{d}W_s\right)   \big|{\cal F}_t   \right]}{ \Vert \varphi\Vert^2_{[T,T_0]} \mathbb{E}\sqrt{\Pi^*_a(0,T)\mathbb{E}[\delta_y(Y_0)|{\cal F}_T]}\sqrt{\Pi^*_a(0,t)}}\varphi_t\\
&\quad +\frac{1}{2} \frac{X_0\mathbb{E}\left[ \sqrt{\Pi^*_a(t,T)\mathbb{E}[\delta_y(Y_0)|{\cal F}_T]  } |{\cal F}_t   \right]}{  \mathbb{E}\sqrt{\Pi^*_a(0,T)\mathbb{E}[\delta_y(Y_0)|{\cal F}_T]}\sqrt{\Pi^*_a(0,t)}}\iota_t.
\end{aligned}
\end{equation}

Substituting (\ref{iknow1})  into (\ref{bsde_barzt1}) we obtain the robust optimal portfolio strategy
 \begin{equation}\label{bsde_barzt1_donsk}
   \begin{aligned}
 \pi_t^*&=\pi_t^*(y)|_{y=Y_0}\\
 & =\frac{\iota_t}{2\sigma_t}+\frac{1}{2}\frac{\mathbb{E}\left[\tilde \Pi^*_a(t,T) \sqrt{  \mathbb{E}[\delta_y(Y_0)|{\cal F}_T]   } \left(y-\int_0^T\varphi_s\mathrm{d}W_s\right)   \big|{\cal F}_t   \right]}{\sigma_t \Vert \varphi\Vert^2_{[T,T_0]}\mathbb{E}\left[\tilde \Pi^*_a(t,T)\sqrt{ \mathbb{E}[\delta_y(Y_0)|{\cal F}_T]} |{\cal F}_t  \right]}\varphi_t\Bigg|_{y=Y_0}  ,
 \end{aligned} 
 \end{equation}
 where
  \begin{equation}  
  \begin{aligned}
\tilde\Pi^*_a(0,t):= \exp\Bigg\{  -\int_{0}^{t} \frac{\iota_s}{2}  \mathrm{d}W_s
-\frac{1}{8}\int_{0}^{t}    \iota_s ^2 \mathrm{d}s \Bigg\}.
  \end{aligned}
 \end{equation}
Then by the Girsanov theorem,  $W_{\mathbb Q}(t):=W_t+\int_0^t\frac{\iota_s}{2}\mathrm{d}s$ is an ${\cal F}_t$-Brownian motions under the new equivalent probability measure $\mathbb Q$ defined by $\mathrm{d}{\mathbb Q}=\tilde \Pi^*_a(0,T)\mathrm{d}\mathbb{P}$. Thus, by the Bayes rule (see \cite{Karatzas91}), we can rewritten the robust optimal portfolio strategy as follows
 \begin{equation} \label{markov_pi}
   \begin{aligned}
 \pi_t^* =&\frac{\iota_t}{2\sigma_t}+\frac{1}{2}\frac{\mathbb{E}_{\mathbb Q}\left[  \sqrt{  \mathbb{E}[\delta_y(Y_0)|{\cal F}_T]   } \left(y-\int_0^T\varphi_s\mathrm{d}W_s\right)   \big|{\cal F}_t   \right]}{\sigma_t \Vert \varphi\Vert^2_{[T,T_0]}\mathbb{E}_{\mathbb Q}\left[ \sqrt{ \mathbb{E}[\delta_y(Y_0)|{\cal F}_T]} |{\cal F}_t  \right]}\varphi_t  \Bigg|_{y=Y_0}\\
  =&\frac{\iota_t}{2\sigma_t}+\frac{1}{2}\frac{\varphi_t}{\sigma_t \Vert \varphi\Vert^2_{[T,T_0]}} \Bigg\{y-
\frac{\mathbb{E}_{\mathbb Q}\Big[ \exp\Big\{-\frac{(y-\int_0^T\varphi_s\mathrm{d}W_{\mathbb Q}(s)+\tilde\iota_T)^{ 2}}{4\Vert\varphi \Vert_{[T,T_0]}^2}  \Big\}\Big( \int_0^T\varphi_s\mathrm{d}W_{\mathbb Q}(s)  -\tilde\iota_T\Big)  \big|{\cal F}_t   \Big]}{\mathbb{E}_{\mathbb Q}\Big[ \exp\Big\{-\frac{(y-\int_0^T\varphi_s\mathrm{d}W_{\mathbb Q}(s)+\tilde\iota_T)^{ 2}}{4\Vert\varphi \Vert_{[T,T_0]}^2}  \Big\} \big|{\cal F}_t  \Big]}\Bigg\}   \Bigg|_{y=Y_0},
\end{aligned} 
 \end{equation}
where $\tilde\iota_t:=\frac{1}{2}\int_0^t\varphi_s\iota_s\mathrm{d}s$, $t\in[0,T]$.
On the other hand, the conditional ${\mathbb Q}$ law of $\int_0^T\varphi_s\mathrm{d}W_{\mathbb Q}(s)$, given ${\cal F}_t$, is normal with mean $\int_0^t\varphi_s\mathrm{d}W_{\mathbb Q}(s)$ and variance $\Vert\varphi\Vert^2_{[t,T]}$ due to the Markov property of It\^{o} diffusion processes (see \cite{Karatzas91}). Thus (\ref{markov_pi}) leads to 
 \begin{equation}\label{bsde_barzt1_donsk01}  
   \begin{aligned}
 \pi_t^*  
  =&\frac{\iota_t}{2\sigma_t}+\frac{1}{2}\frac{\varphi_t}{\sigma_t \Vert \varphi\Vert^2_{[T,T_0]}} \Bigg\{y+\tilde \iota_T-
\frac{\int_{\mathbb{R}}   \frac{x}{\sqrt{2\pi\Vert\varphi \Vert_{[t,T]}^2}}   \exp\Big\{-\frac{(y -x+\tilde\iota_T)^{ 2}}{4\Vert\varphi \Vert_{[T,T_0]}^2} -\frac{(x-z)^2}{2\Vert\varphi \Vert_{[t,T]}^2} \Big\}    \mathrm{d}x}
{ \int_{\mathbb{R}}  \frac{1}{\sqrt{2\pi\Vert\varphi \Vert_{[t,T]}^2}}    \exp\Big\{-\frac{(y -x+\tilde\iota_T)^{ 2}}{4\Vert\varphi \Vert_{[T,T_0]}^2} -\frac{(x-z)^2}{2\Vert\varphi \Vert_{[t,T]}^2} \Big\}    \mathrm{d}x    }\Bigg\}    \Bigg|_{\substack{z=\int_0^t\varphi_s\mathrm{d}W_{\mathbb Q}(s) \\ y=Y_0}}\\
=&\frac{\iota_t}{2\sigma_t}+\frac{1}{2}\frac{\varphi_t}{\sigma_t \Vert \varphi\Vert^2_{[T,T_0]}} \Bigg\{y+\tilde \iota_T-
 \frac{\Vert\varphi \Vert_{[t,T]}^2(y+\tilde\iota_T)+2\Vert\varphi \Vert_{[T,T_0]}^2 z}{\Vert\varphi \Vert_{[t,T]}^2+2\Vert\varphi \Vert_{[T,T_0]}^2}
 \Bigg\}   \Bigg|_{\substack{z=\int_0^t\varphi_s\mathrm{d}W_{\mathbb Q}(s) \\ y=Y_0}}\\
=&\frac{\iota_t}{2\sigma_t}+\frac{ Y_0-\int_0^t\varphi_s\mathrm{d}W_s+\frac{1}{2}\int_t^T\varphi_s\iota_s\mathrm{d}s }{\sigma_t\left(\Vert\varphi \Vert_{[t,T_0]}^2+\Vert\varphi \Vert_{[T,T_0]}^2\right)}\varphi_t   .
\end{aligned} 
 \end{equation}

By the equation (\ref{halfequ_21}) in Theorem \ref{mainth6}, we can also obtain 
 \begin{equation}\label{bsde_barzt1_donsktheta}
   \begin{aligned}
 &\theta^*_t=-\frac{1}{2} \iota_t+\frac{ Y_0-\int_0^t\varphi_s\mathrm{d}W_s+\frac{1}{2}\int_t^T\varphi_s\iota_s\mathrm{d}s }{\left(\Vert\varphi \Vert_{[t,T_0]}^2+\Vert\varphi \Vert_{[T,T_0]}^2\right)}\varphi_t-\frac{Y_0-\int_0^t\varphi_s\mathrm{d}W_s}{\Vert \varphi\Vert^2_{[t,T_0]}}\varphi_t.
\end{aligned} 
 \end{equation}

To sum up, we give the following theorem for this particular case.

    \begin{theorem} \label{ex_mainth_idontknow_donsker}
Assume that $\mu(t,x)=\mu_0(t) $ for some bounded function $\mu_0(t)$, and $g(\theta)=\frac{1}{2}  \theta^2 $. Assume further that $\{{\cal H}_t\}_{0\le t\le T}$ is given by (\ref{donsker_filtration}) and all parameter processes are deterministic bounded functions. Suppose $(\pi^*,\theta^*)\in {\cal A}_1'\times{\cal A}_2'$ is optimal for Problem \ref{sdg} under the conditions of Theorem \ref{mainth4}. Then $(\pi^*,\theta^*)$ is given by (\ref{bsde_barzt1_donsk01}) and (\ref{bsde_barzt1_donsktheta}).
 \end{theorem}
 
 Moreover, if $\varphi=1$, $\pi^*$ is given by
  \begin{equation}\label{cdonsu}
   \begin{aligned}
  \pi^*_t&=\frac{\mu_0(t)-r_t}{2\sigma^2_t}+\frac{ W_{T_0}-W_t+\frac{1}{2}\int_t^T\frac{\mu_0(s)-r_s}{\sigma_s}\mathrm{d}s  }{\sigma_t \left( T_0-t+T_0-T\right)}    ,
 \end{aligned} 
 \end{equation}
 and $\theta^*$ is given by
   \begin{equation}\label{cdonsv}
   \begin{aligned}
&\theta_t^*=-\frac{\mu_0(t)-r_t}{2\sigma_t}+ \frac{ W_{T_0}-W_t+\frac{1}{2}\int_t^T\frac{\mu_0(s)-r_s}{\sigma_s}\mathrm{d}s  }{ \left( T_0-t+T_0-T\right)}-\frac{W_{T_0}-W_t}{ T_0-t} . 
 \end{aligned} 
 \end{equation}

By (\ref{c3determined_donsker}) we have 
 \begin{equation}  \label{compareX}
  \begin{aligned}
\mathbb{E}(\sqrt{\Pi^*(0,T)}|{\cal H}_0)= \left(\mathbb{E}\sqrt{\Pi^*_a(0,T)\frac{\mathbb{E}[\delta_y(Y_0)|{\cal F}_T]}{\mathbb{E}[\delta_y(Y_0)|{\cal F}_0]}}\right) \Bigg{|}_{y=Y_0}\end{aligned}
 \end{equation} 
Substituting (\ref{compareX}) into (\ref{dontknowv}), we have by Girsanov theorem and tedious calculation that
 \begin{equation}  \label{value2}
  \begin{aligned}
  V=&\ln X_0-2\mathbb{E}\left[\left( \ln  \mathbb{E}\sqrt{\Pi^*_a(0,T)\frac{\mathbb{E}[\delta_y(Y_0)|{\cal F}_T]}{\mathbb{E}[\delta_y(Y_0)|{\cal F}_0]}}\right)\Bigg{|}_{y=Y_0}  \right]
  \\=& \ln X_0+ \int_0^T r_t \mathrm{d}t  +\frac{1}{4}\int_0^T\left( \frac{\mu_0(t)-r_t}{\sigma_t}  \right)^2\mathrm{d}t  -2\mathbb{E}\left[  \left(\ln \mathbb{E}_{\mathbb Q}\sqrt{\frac{\mathbb{E}[\delta_y(Y_0)|{\cal F}_T]}{\mathbb{E}[\delta_y(Y_0)|{\cal F}_0]}   } \right)\Bigg{|}_{y=Y_0}    \right]\\
  =& \ln X_0+ \int_0^T r_t \mathrm{d}t  +\frac{1}{4}\int_0^T \left( \frac{\mu_0(t)-r_t}{\sigma_t}  \right)^2 \mathrm{d}t +\frac{1}{2}\ln \left( 1-\frac{T^2}{(2T_0-T)^2} \right)^{-1}+\frac{T}{2(2T_0-T)}\\
  &  +\frac{1}{4(2T_0-T)}\left( \int_0^T \frac{\mu_0(t)-r_t}{\sigma_t} \mathrm{d}t  \right)^2
   \end{aligned}
 \end{equation}

Thus we have the following corollary.
 
     \begin{corollary} \label{ex_mainth_idontknow_donskerc}
Assume that $\mu(t,x)=\mu_0(t) $ for some bounded function $\mu_0(t)$, and $g(\theta)=\frac{1}{2} \theta^2$. Assume further that $\{{\cal H}_t\}_{0\le t\le T}$ is given by (\ref{donsker_filtration}) with $\varphi=1$ and all parameter processes are bounded functions. Suppose $(\pi^*,\theta^*)\in {\cal A}_1'\times{\cal A}_2'$ is optimal for Problem \ref{sdg} under the conditions of Theorem \ref{mainth4}. Then $(\pi^*,\theta^*)$ is given by (\ref{cdonsu}) and (\ref{cdonsv}), and the value $V$ is given by (\ref{value2}).
 \end{corollary}

\section{The large insider case: combined method}
\label{sec_example1} 
Now we consider the case when the insider is `large', which means that the mean rate of return $\mu$ could be influenced by her portfolio selection $\pi$.

  Just as in Section \ref{sec_exampleg_sub1}, we assume that the mean rate of return $\mu(t,x)=\mu_0(t)+\varrho_tx$ for some ${\cal F}_t$-adapted measurable processes $\mu_0(t)$ and $\varrho_t$ with $0\le \varrho_t<\frac{1}{2}\sigma_t^2$. Put $\iota_t=\frac{\mu_0(t)-r_t}{\sigma_t}$, $\tilde\sigma_t=\sigma_t-\frac{2\varrho_t}{\sigma_t}$, and $\tilde \phi_t=\iota_t+\phi_t$. Assume further the penalty function $g$ is given by $g(\theta)=\frac{1}{2}\theta^2$. Then we have
 \begin{equation}\label{utilityofe}
 \mathbb{E}\left[  \int_0^T\varepsilon_s^{\theta^*}  g(\theta^*_s)\mathrm{d}s \right]
= \mathbb{E}\left[\varepsilon^{\theta^*}_T  \ln \varepsilon^{\theta^*}_T  \right].
 \end{equation}
 
In this setting, the equation (\ref{halfequ_21}) in Theorem \ref{mainth6} can be reduced to
 \begin{equation}\label{reduceth6} 
 \begin{aligned}
 \theta_t^*=\tilde \sigma_t\pi^*_t-\tilde\phi_t. 
  \end{aligned}
 \end{equation}

By a similar procedure in Section \ref{sec_exampleg_sub1} with respect to the Hamiltonian system (\ref{half_equ3}), we have (see (\ref{equeg1_42fb}))
 \begin{equation}  \label{equeg1_42}
  \begin{aligned}
\ln\big( \varepsilon^{\theta^*}_TX_T^{\pi^*}\big)=c^*_2,
 \end{aligned}
 \end{equation} 
where $c^*_2=p^*_2(0)$ is an ${\cal H}_0$-measurable random variable.

The It\^{o} formula for It\^{o} integrals combined with the expressions of $\varepsilon_t^{\theta^*}$ and $X^{\pi^*}_t$ yields the following SDE:
     \begin{equation}\label{bsde_lt} 
 \begin{aligned}
\mathrm{d}\ln\big(\varepsilon^{\theta^*}_tX^{\pi^*}_t\big)=&\left[r_t+(\mu_0(t)-r_t)\pi_t^*+ \sigma_t\pi_t^*  \phi_t   \right]\mathrm{d}t -\frac{1}{2}\left[  (\sigma_t\pi^*_t)^2 +(\theta^*_t)^2-\sigma_t(\sigma_t-\tilde\sigma_t)(   \pi^*_t)^2 \right]\mathrm{d}t\\
&+\left(\sigma_t\pi^*_t+\theta^*_t\right)\mathrm{d}W_{\cal H}(t).
\end{aligned} 
 \end{equation}
Put $L^*_t=\ln\big(\varepsilon^{\theta^*}_tX^{\pi^*}_t\big)$ and $ z^*_t= \sigma_t\pi^*_t+\theta^*_t$. Thus, combining it with (\ref{reduceth6}) we have
   \begin{equation}\label{bsde_zt1}  
   \begin{aligned}
 &\pi_t^*=\frac{ z_t^*+\tilde\phi_t}{\sigma_t+\tilde\sigma_t}  ,   
\end{aligned} 
 \end{equation}
and
   \begin{equation}\label{bsde_zt2} 
\begin{aligned}
     &\theta^*_t=\frac{\tilde\sigma_tz^*_t-\sigma_t\tilde\phi_t}{\sigma_t+\tilde\sigma_t}. 
\end{aligned} 
 \end{equation}
The SDE (\ref{bsde_lt}) combined with (\ref{bsde_zt1}), (\ref{bsde_zt2}) and (\ref{equeg1_42}) leads to the following quadratic BSDE  
 \begin{eqnarray}\label{final_ex_1}
  \left\{ \begin{aligned}&\mathrm{d}L^*_t=-f_{\text{Q}}(t, z^*_t,\omega)\mathrm{d}t+  z^*_t\mathrm{d}W_{\cal H}(t), \quad 0\le t\le T,
  \\&L^*_T=c^*_2,\end{aligned} \right.   
 \end{eqnarray}
where the generator $f_{\text{Q}}:[0,T]\times \mathbb{R}\times\Omega\rightarrow \mathbb{R}$  is given by
  \begin{equation*} 
\begin{aligned}
f_{\text{Q}}(t,z,\omega)=&  \frac{z^2}{4}  -\frac{\tilde\phi_t}{2}z     -r_t
-  \frac{\tilde\phi_t^2}{4}  -\frac{(\sigma_t-\tilde\sigma_t)}{4(\sigma_t+\tilde\sigma_t)} (z+\tilde\phi_t)^2 .
\end{aligned} 
 \end{equation*}
 
  By (\ref{utilityofe}) and (\ref{equeg1_42}), the value $V$ can be calculated by
\begin{equation}\label{quavalue}
\begin{aligned}
V=\mathbb{E}\left[\varepsilon^{\theta^*}_T\ln\left(\varepsilon^{\theta^*}_T X^{\pi^*}_T  \right)  \right]=\mathbb{E} c^*_2=\mathbb{E}L^*_T.
\end{aligned}\end{equation}

If the filtration $\{{\cal H}_t\}_{0\le t\le T}$ is the augmentation of the natural filtration of $W_{\cal H}(t)$, which was also assumed in \cite{Biagini05}, then $c^*_2$ is a constant. Suppose that $c_2^*$, $\tilde\phi$, $r$, $\sigma$ and $\tilde\sigma$ are bounded. Then, by \cite[Theorems 4.1]{Fujii18}, the quadratic BSDE (\ref{final_ex_1}) has a unique strong solution $(L^*,  z^*)$, i.e., $L^*_t$ is a bounded continuous ${\cal H}_t$-adapted process, $  z^*_t$ is a measurable ${\cal H}_t$-adapted process with $\mathbb{E}\int_0^T |  z^*_t|^2\mathrm{d}t<\infty$ and $\int_0^t  z^*_s\mathrm{d}W_{\cal H}(s)$ is an ${\cal H}_t$-${\cal BMO}$-martingale (see \cite{He92}), and $(L^*,   z^*)$ satisfies the BSDE (\ref{final_ex_1}). 

\begin{remark}
Under mild conditions on the Malliavin derivative, we can obtain the formulae for $   z^*_t $ as follows (see \cite[Corollary 5.1]{Fujii18})
 \begin{eqnarray}\label{final_ex_bmoofz}
  \begin{aligned}
   z^*_t= D_t L_t^*.
  \end{aligned}     
 \end{eqnarray}
\end{remark}

To sum up, we give the following theorem.

   \begin{theorem} \label{ex_mainth1}
Assume that $\mu(t,x)=\mu_0(t)+\varrho_tx$ for some ${\cal F}_t$-adapted measurable processes $\mu_0(t)$ and $\varrho_t$ with $0\le \varrho_t\le \frac{1}{2}\sigma_t^2$, and $g(\theta)=\frac{1}{2} \theta^2 $. Suppose $(\pi^*,\theta^*)\in {\cal A}_1'\times{\cal A}_2'$ is optimal for Problem \ref{sdg} under the conditions of Theorem \ref{mainth4}. Then $(\pi^*,\theta^*)$ and $V$ are given by (\ref{bsde_zt1}), (\ref{bsde_zt2}) and (\ref{quavalue}), where $(L^*,  z^*)$ solves the quadratic BSDE (\ref{final_ex_1}), where the ${\cal H}_0$-measurable random variable $c^*_2$ can be determined by the initial value condition $L^*_0=\ln X_0$\footnote{In fact, integrating (\ref{final_ex_1}) from $t$ to $T$ yields $L^*_T-L^*_t=-\int_t^Tf_{\text{Q}}(s,  z^*_s,\omega)\mathrm{d}s+\int_t^T  z^*_s\mathrm{d}W_{\cal H}(s)$. Taking conditional expectation and assuming the It\^{o} integrals are $L^2$-martingales, we get $L^*_t=\mathbb{E}\left[  \int_t^Tf_{\text{Q}}(s,  z^*_s,\omega)\mathrm{d}s+L^*_T \big{|}{\cal H}_t\right]$. Taking $t=0$ and using the initial value condition we have $c^*_2=\ln X_0-\mathbb{E}\left[  \int_0^Tf_{\text{Q}}(s,   z^*_s,\omega)\mathrm{d}s \big{|}{\cal H}_0\right]$.}. Furthermore, if $\{{\cal H}_t\}_{0\le t\le T}$ is the augmentation of the natural filtration of $W_{\cal H}(t)$, and $c_2^*$, $\tilde\phi$, $r$, $\sigma$ and $\tilde\sigma$ are bounded, then the quadratic BSDE (\ref{final_ex_1}) has a unique strong solution and $z^*$ is given by (\ref{final_ex_bmoofz}) under mild conditions, where $c^*_2$ can be determined by traversing all constants such that the condition $L^*_0=\ln X_0$ holds.
 \end{theorem}

\begin{remark}
If the filtration $\{{\cal H}_t\}_{0\le t\le T}$ in Theorem \ref{ex_mainth1} is not  the augmentation of the natural filtration of $W_{\cal H}(t)$, or the coefficients of the generator $f_{\text{Q}}$ is not necessarily bounded, we refer to \cite{Draouil15,Eyraoud05,Li06,Lu13,Wang07} for further results. Meanwhile, the ${\cal H}_0$-measurable random variable $c^*_2$ can be determined by traversing all ${\cal H}_0$-measurable random variable such that the condition $L^*_0=\ln X_0$ holds. Moreover, if ${\cal H}_0$ is generated by a random variable $F$ and all $\mathbb{P}$-negligible sets, then by the monotone class theorem of functional forms (see \cite{He92}), there exists a Borel measurable function $f$ such that $c^*_2=f(F)$, a.s. Thus, $c^*_2$ can be determined by traversing all Borel measurable functions $f$ such that the initial value condition $L^*_0=\ln X_0$ holds. 
\end{remark}

\subsection{A particular case}\label{Sub_particular1_d}
Next, we consider a particular case when the filtration is of initial enlargement type, i.e.,  
\begin{equation}\label{filtp}
{\cal H}_t=\bigcap_{s>t}({\cal F}_s\vee Y_0):=\bigcap_{s>t}\left({\cal F}_s\vee \int_0^{T_0}\varphi_{u}\mathrm{d}W_{u}\right),\quad 0\le t\le T,
\end{equation}
for some $T_0>T$, and all the parameter processes are assumed to be deterministic bounded functions. Here, $\varphi_t $ is some deterministic function satisfying $\Vert \varphi\Vert^2_{[s,t]}:=\int_s^{t}\varphi_u^2    \mathrm{d}u<\infty$ for all $0\le s\le t\le T_0$, and $\Vert \varphi\Vert^2_{[T,T_0]}>0$.

By the Donsker $\delta$ functional $\delta_y(Y_0)$ and similar procedure in Section \ref{Sub_particular1}, we have
\begin{equation} \label{dons_phi12}
\begin{aligned}
&\phi_t=\phi_t(y)|_{y=Y_0}=\frac{y-\int_0^t\varphi_s\mathrm{d} W_s}{\Vert \varphi\Vert^2_{[t,T_0]}}  \varphi_t\Bigg{|}_{y=Y_0}.
  \end{aligned}\end{equation}
Thus the BSDE (\ref{final_ex_1}) is equivalent to the following classical quadratic BSDE with respect to the filtration $\{{\cal F}_t\}_{0\le t\le T}$
 \begin{eqnarray}\label{final_ex_1dons}
  \left\{ \begin{aligned}&\mathrm{d}L^*_t(y)=-\bar f_{\text{Q}}(t, z^*_t(y),y,\omega)\mathrm{d}t+ z^*_t(y)\mathrm{d}W_t, \quad 0\le t\le T,
  \\&L^*_T(y)=c^*_2(y),\end{aligned} \right.   
 \end{eqnarray}
where the generator $\bar f_{\text{Q}}:[0,T]\times \mathbb{R} \times\mathbb{R}\times\Omega\rightarrow \mathbb{R}$  is given by
  \begin{equation*} 
\begin{aligned}
&\bar f_{\text{Q}}(t,z,y,\omega)=  \frac{z^2}{4}  -\frac{\iota_t-\phi_t(y)}{2}z   -r_t
-  \frac{(\iota_t+\phi_t(y))^2 }{4}  -\frac{ (\sigma_t-\tilde\sigma_t)}{4(\sigma_t+\tilde\sigma_t ) } 
  \left(z +\iota_t+\phi_t(y)  \right)^2 .
\end{aligned} 
 \end{equation*}

Moreover, the value can be calculated by
\begin{equation}\label{quavalue2}
V=\mathbb{E}(L^*_T(y)|_{y=Y_0}).
\end{equation}

Thus we have the following result.

   \begin{theorem} 
Assume that $\mu(t,x)=\mu_0(t)+\varrho_tx$ for some bounded functions $\mu_0(t)$ and $\varrho_t$ with $0\le \varrho_t< \frac{1}{2}\sigma_t^2$, and $g(\theta)=\frac{1}{2}  \theta^2 $. Assume further that $\{{\cal H}_t\}_{0\le t\le T}$ is given by (\ref{filtp}) and all parameter processes are deterministic bounded functions. Suppose $(\pi^*,\theta^*)\in {\cal A}_1'\times{\cal A}_2'$ is optimal for Problem \ref{sdg} under the conditions of Theorem \ref{mainth4}. Then $(\pi^*,\theta^*)$ and $V$ are given by (\ref{bsde_zt1}), (\ref{bsde_zt2}) and (\ref{quavalue2}), where $(L^*(y), z^*(y) )$ solves the classical quadratic BSDE (\ref{final_ex_1dons}), $\phi_t(y)$ is given by (\ref{dons_phi12}), and the ${\cal H}_0$-measurable random variable $c^*_2$ can be determined by traversing all Borel measurable functions $c^*_2(y)$ such that $L^*_0=\ln X_0$.
 \end{theorem}
 
\section{The large insider case without model uncertainty}\label{Sub_particular2}
When considering the large insider case, the quadratic BSDE (\ref{final_ex_1}) or (\ref{final_ex_1dons}) has no closed-form solution. Thus we concentrate on the special situation without model uncertainty, that is, ${\cal A}_2'=\{(0)\}$. Then Problem \ref{sdg} degenerates to the following anticipating stochastic control problem.

\begin{problem}\label{sdg2}
 Select $\pi^*\in {\cal A}_1' $ such that
\begin{equation} \label{valueofsdg2}
 \tilde V:=\tilde J(\pi^*)=\sup_{\pi\in {\cal A}_1'}   \tilde J(\pi),
\end{equation}  
where $\tilde J(\pi):=\mathbb{E}\left[  \ln X_T^\pi \right]$. We call $\tilde V$ the value (or the optimal expected utility) of Problem \ref{sdg2}.
\end{problem}
 
Suppose that $\mu(t,x)=\mu_0(t)+\varrho_tx$ for some ${\cal F}_t$-adapted measurable processes $\mu_0(t)$ and $\varrho_t$ with $0\le \varrho_t<\frac{1}{2}\sigma_t^2$, and Assumptions \ref{assump1} and \ref{assump2} hold. Put $\iota_t=\frac{\mu_0(t)-r_t}{\sigma_t}$, $\tilde\sigma_t=\sigma_t-\frac{2\varrho_t}{\sigma_t}$, and $\tilde \phi_t=\iota_t+\phi_t$. 
  
 Since ${\cal A}_2'=\{0\}$, which implies that $\theta^*=0$ in (\ref{reduceth6}). Thus, the optimal strategy $\pi^*\in{\cal A}'_1$ is given by
 \begin{equation}\label{reduceth6_2} 
 \begin{aligned}
 \pi^*_t=\frac{\tilde \phi_t}{\tilde\sigma_t}=\frac{\mu_0(t)-r_t }{\sigma_t\tilde\sigma_t}+\frac{ \phi_t}{\tilde\sigma_t}.  \end{aligned}
 \end{equation}
 
 Substitute (\ref{reduceth6_2}) into (\ref{valueofsdg2}). Then by (\ref{halfequ_21}) and tedious calculation we have
  \begin{equation}\label{value3} 
  \begin{aligned}
   \tilde V=&\ln X_0+\mathbb{E}\int_0^Tr_t\mathrm{d}t+\frac{1}{2}\mathbb{E}\int_0^T
  \frac{ \sigma_t}{ \tilde\sigma_t }\left(\frac{\mu_0(t)-r_t}{\sigma_t}+\phi_t\right)^2   \mathrm{d}t.
    \end{aligned}
 \end{equation}

 Assume further that the insider filtration $\{{\cal H}_t\}_{0\le t\le T}$ is given by
  \begin{equation} \label{filtrationex_conti}
 \begin{aligned}
 {\cal F}_t\subset {\cal H}_t\subset  \bigcap_{s>t}\left({\cal F}_s\vee  W_{T_0}\right)=:\bar{\cal H}_t, \quad 0\le t\le T,  
   \end{aligned}
 \end{equation}
 for some $T_0>T$.
Then the enlargement of filtration technique can be applied to give the concrete expression of $\tilde \phi$ in (\ref{reduceth6_2}). We give the following lemma to characterize the decomposition of the ${\cal H}_t$-semimartingale $W$ in Theorem \ref{mainth3}. The proof can be found in \cite[page 327]{DiNunno09}.

\begin{lemma}
[Enlargement of filtration] 
\label{lemofenlarge1}
The process $W_t$, $t\in[0,T]$, is a semimartingale with respect to the filtration $\{{\cal H}_t\}_{0\le t\le T}$ given by (\ref{filtrationex_conti}). Its semimartingale decomposition is 
  \begin{equation}  
 \begin{aligned}
W _t=W_{\cal H}(t)+\int_0^t\frac{\mathbb{E}\big[W_{T_0}|  {\cal H}_s\big]-W_s}{T_0-s}\mathrm{d}s,\quad 0\le t\le T,  
   \end{aligned}
 \end{equation}
\end{lemma}
where $W_{\cal H}(t)$ is an ${\cal H}_t$-Brownian motion.

Combined with Lemma \ref{lemofenlarge1}, the ${\cal H}_t$-adapted process $\phi_t$ in Theorem \ref{mainth3} is of the form
   \begin{equation}  
 \begin{aligned}
 \phi_t=\frac{\mathbb{E}\big[W_{T_0}|{\cal H}_t\big]-W_t}{T_0-t},\quad 0\le t\le T.
   \end{aligned}
 \end{equation}
 Thus, we can give a concrete characterization of the optimal strategy $\pi^*$ as follows
  \begin{equation}\label{reduceth6_2_2} 
 \begin{aligned}
 &\pi^*_t=\frac{  \mu_0(t)-r_t}{\sigma_t\tilde\sigma_t}+\frac{   \mathbb{E}\big[W_{T_0}|{\cal H}_t\big]-W_t}{\tilde\sigma_t({T_0-t})},
  \end{aligned}
 \end{equation}
 which coincides with the result (16.162) in \cite{DiNunno09}.
 
 To sum up, we give the following theorem.
 
    \begin{theorem} \label{ex_mainth1_special}
Assume that $\mu(t,x)=\mu_0(t)+\varrho_tx$ for some ${\cal F}_t$-adapted measurable processes $\mu_0(t)$ and $\varrho_t$ with $0\le \varrho_t<\frac{1}{2}\sigma_t^2$, and no model uncertainty is considered. Suppose $\pi^* \in {\cal A}_1' $ is optimal for Problem \ref{sdg2} under Assumptions \ref{assump1} and \ref{assump2}. Then $\pi^* $ is given by (\ref{reduceth6_2}), and $\tilde V$ is given by (\ref{value3}). Furthermore, if the filtration $\{{\cal H}_t\}_{0\le t\le T}$ is of the form (\ref{filtrationex_conti}), then $\pi^*$ is given by (\ref{reduceth6_2_2}).
 \end{theorem}

Suppose all parameter processes are deterministic bounded functions. When ${\cal H}_t={\cal F}_t$ (i.e., the investor has no insider information) and ${\cal H}_t=\bar{\cal H}_t$ given by (\ref{filtrationex_conti}) (i.e., the investor owns the insider information $ W_{T_0}$ about the future price of the risky asset), we have the following results.

   \begin{corollary}  \label{endcor1}
Assume that $\mu(t,x)=\mu_0(t)+\varrho_tx$ for some deterministic bounded functions $\mu_0(t)$ and $\varrho_t$ with $0\le \varrho_t<\frac{1}{2}\sigma_t^2$, and no model uncertainty is considered. Assume further that ${\cal H}_t=  {\cal F}_t$ and all parameter processes are deterministic bounded functions. Suppose $\pi^* \in {\cal A}_1' $ is optimal for Problem \ref{sdg2} under the conditions of Theorem \ref{mainth2}. Then $\pi^*$ and $\tilde V$ are given by 
  \begin{equation} 
 \begin{aligned}
 &\pi^*_t=    \frac{\mu_0(t)-r_t}{\sigma_t  \tilde\sigma_t  },
  \\&\tilde V=\ln X_0+\int_0^Tr_t\mathrm{d}t+ \frac{1}{2}\int_0^T \frac{\sigma_t}{ \tilde\sigma_t}\left(\frac{\mu_0(t)-r_t}{\sigma_t}\right)^2\mathrm{d}t.
  \end{aligned}
 \end{equation}
 \end{corollary}

  \begin{corollary}  \label{endcor2}
Assume that $\mu(t,x)=\mu_0(t)+\varrho_tx$ for some deterministic bounded functions $\mu_0(t)$ and $\varrho_t$ with $0\le \varrho_t< \frac{1}{2}\sigma_t^2$, and no model uncertainty is considered. Assume further that ${\cal H}_t=\bar{\cal H}_t$ and all parameter processes are deterministic bounded functions. Suppose $\pi^* \in {\cal A}_1' $ is optimal for Problem \ref{sdg2} under the conditions of Theorem \ref{mainth2}. Then $\pi^*$ and $\tilde V$ are given by 
  \begin{equation} \label{endoptimal}
 \begin{aligned}
 &\pi^*_t=\frac{  \mu_0(t)-r_t}{\sigma_t\tilde\sigma_t}+\frac{   W_{T_0} -W_t}{\tilde\sigma_t({T_0-t})} ,
  \\&\tilde V=\ln X_0+\int_0^Tr_t\mathrm{d}t+\frac{1}{2}\int_0^T   \frac{ \sigma_t}{\tilde\sigma_t}  \left(\frac{\mu_0(t)-r_t}{\sigma_t}\right)^2\mathrm{d}t+\frac{1}{2}\int_0^T  \frac{ \sigma_t}{\tilde\sigma_t}  \frac{1}{T_0-t}  \mathrm{d}t.
  \end{aligned}
 \end{equation}
 \end{corollary}

\section{Numerical analysis}
\label{numericalsec}

\begin{figure}[htbp]
  \centering
  \includegraphics[scale=0.6]{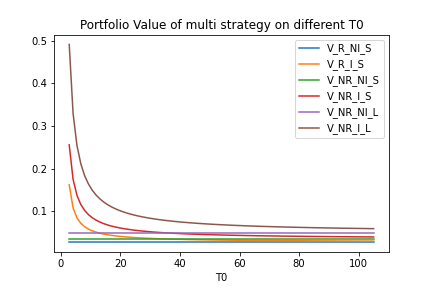}
  \caption{Portfolio Value of multi strategy on different T0. $\mu=0.15,\sigma=0.35,T=1,\rho= \frac{1}{4} \sigma^{2}$}
  \label{fig.1}
\end{figure}

We standardize the paramter T to 1.
As $T_{0}$ increase , the additional information of the insider decrease and the profit from insider trading of insider trading decays. The information rent is $\frac{1}{2}\int_0^T  \frac{ \sigma_t}{\tilde\sigma_t}  \frac{1}{T_0-t}  \mathrm{d}t.$ for large trader without model uncertainty , $\frac{1}{2}\int_0^T    \frac{1}{T_0-t}  \mathrm{d}t.$ for small trader and $\frac{1}{2}\ln \left( 1-\frac{T^2}{(2T_0-T)^2} \right)^{-1}+\frac{T}{2(2T_0-T)}+\frac{1}{4(2T_0-T)}\left( \int_0^T \frac{\mu_0(t)-r_t}{\sigma_t} \mathrm{d}t  \right)^2$ for robust
small trader, which is approximately inversely proportional to $\frac{1}{T_0-T}$. If $T_0$ is 10 times larger than
T , the outperformance of insider portfolio is economically insignificant.In simulation., the parameters are set as follows: $\mu=0.15,\sigma=0.35,T=1,\rho= \frac{1}{4} \sigma^{2}$.

    Robust portfolio reduces risk of model uncertainty  at the cost of  absolute return. However, if T0 is smaller than 12 T , the small robust insider still outperforms the small non-robust dealers. In other words,
the profit from insider trading offsets the loss from model uncertainty from model uncertainty.

To view this more clear, we plot the insider information needed,that is $ T_{0}$, on drift part $\mu$ and volatility $\sigma$.If the $\mu$ is larger or the $\sigma$ is smaller, the loss of model uncertainty will be larger. 
\begin{figure}[htbp]
  \centering
  \includegraphics[scale=0.6]{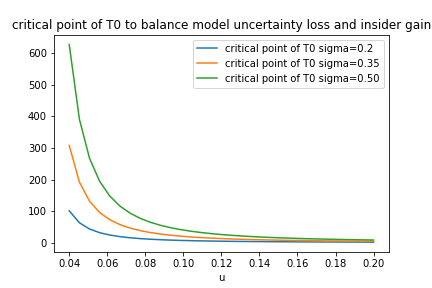}
  \caption{Insider information needed for an ambiguity-averse insurer to gain the same utility of an ambiguity-neutral insurer with no insider information. $T=1,\rho= \frac{1}{4} \sigma^{2},t= \frac{1}{2}T, W_{T_0}=1$}
  \label{fig.1}
\end{figure}

Large trades always obtain more profit than small traders in our models because they could  influence parameter $\mu$.

\begin{figure}[htbp]
  \centering
  \includegraphics[scale=0.6]{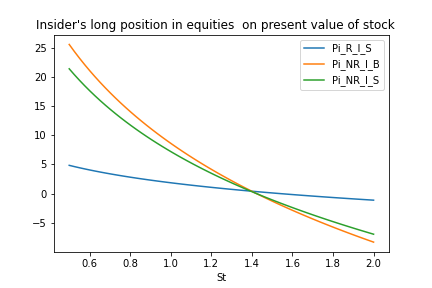}
  \caption{Portfolio Value of multi strategy on different T0. $\mu=0.08,\sigma=0.35,T=1,\rho= \frac{1}{4} \sigma^{2},t= \frac{1}{2}T, W_{T_0}=1$}
  \label{fig.1}
\end{figure}

We also analysis the relationship between optimal stock holdings and present stock price.We only show the curves of insiders because the strategy of traders without inside information will be trivial on the assumption of  constant model parameters.
Traders will definitely reduce long position of stocks with a higher $W_{t}. $ Conditional on $W_{T_0}=W{t}$,the all types of traders will hold long position of stocks because of positive drift term .

The robust trader is dramatically less aggressive and the derivative of $\pi_{t}$ with respect to $W_{t} $ is $\frac {-1}{\sigma_{t}(T_0-t+t_{0}-T)}$.Worrying about risk of model uncertainty , the ambiguity averse trader will not response strongly to changes of
the difference of market conditions and insider information.  The large trader without model uncertaity is the most aggressive one with the 
 derivative of $\pi_{t}$ with respect to $W_{t} $ as  $ \frac{-1}{\tilde\sigma_t (T_0-t)  }$.And the corresponding one of  small trader without model uncertaity is  $ \frac{-1}{\sigma_t (T_0-t)  }$.


\section{Conclusion}
\label{sec:conclusion}

In this paper, we improve some properties of the forward integral and obtain the It\^{o} formula for forward integrals by Malliavin calculus. We consider the optimization problem of insider trading under model uncertainty, and give the characterization of the robust optimal portfolio. We consider two typical cases when the insider is small and large, and give the corresponding BSDEs to characterize the portfolio strategies. In the case of small insider, we obtain the closed form of the portfolio strategy and the value function by using the Donsker $\delta$ functional. We use simulation to show the comparison of portfolio strategies under different situations and give some economic analysis.

For further work, the quadratic BSDE corresponding to the robust optimal portfolio strategy for a large insider and more general utility functions need to be studied. Moreover, extending the optimization problem to other models, like the jump-diffusion model, is also a subject of ongoing research.

\section*{Disclosure statement}
No potential conflict of interest was reported by the authors.

\section*{Funding}

The work is funded by the National Natural Science Foundation of China (No. 72071119).
 
\bibliographystyle{IEEEtran}
\bibliography{malliavin_op}

\begin{appendices}
\section{Some results in Malliavin calculus}\label{appM}

\setcounter{equation}{0} 
\setcounter{theorem}{0}

\renewcommand\theequation{A.\arabic{equation}}
\renewcommand\thetheorem{A.\arabic{theorem}}
\renewcommand\theremark{A.\arabic{remark}}

The following lemma is the chain rule of $D_t$, which is an extension of Lemma A.1 in \cite{Ocone91}.

\begin{lemma}
\label{chainrule}
Let ${\bm X}=(X_1,\cdots, X_n)\in D^{1,p}(\Omega;\mathbb{R}^n)$ and $f\in C^1(\mathbb{R}^n)$ for some $n\in {\mathbb N}_+$ and $p\ge 1$. Assume that  
\begin{equation}
 \|f({\bm X})\|_{L^p(\Omega)}+\left\|  \sum_{l=1}^n \frac{\partial f}{\partial x_l}({\bm X})\cdot D_t { X}_l \right\|_{L^p(\Omega; H)}<\infty.
\end{equation}
Then $f({\bm X})\in D^{1,p}(\Omega)$ and $D_tf({\bm X})= \sum_{l=1}^n \frac{\partial f}{\partial x_l}({\bm X})\cdot D_t { X}_l$.
\end{lemma}
\begin{proof}
The proof is similar with that of Lemma A.1 in \cite{Ocone91}.
\end{proof}

The following two lemmas with respect to $\delta$ are the multiplication formula, and the boundedness property, respectively. 

\begin{lemma}
\label{multidelta}
 \textnormal{(Proposition 1.3.3,  \cite{Nualart06})} 
Let $X\in D^{1,2}(\Omega)$ and ${ u}\in \text{Dom}\ {\delta}$. Suppose that $X{ u}\in L^2(\Omega;H)$. Then $X{ u}\in \text{Dom}\ {\delta}$ and we have 
\begin{equation}\label{prop133}
{\delta}(X{ u})=X\delta{ u}-\int_0^T   D_t X\cdot   u_t\mathrm{d}t
\end{equation}
 provided the right-hand side of (\ref{prop133}) is square integrable.
\end{lemma}

\begin{lemma}
\label{boundeddelta}
 \textnormal{(Proposition 1.3.1,  \cite{Nualart06})} 
 $D^{1,2}(\Omega;H)\subset \text{Dom}\ {\delta}$.   Moreover, $\delta$ is bounded from $D^{1,2}(\Omega;H)$ into $L^2(\Omega)$.
\end{lemma}

The following two propositions characterize the local properties of $D_t$ and $\delta$.

\begin{lemma}
\label{locD}
  \textnormal{(Proposition 3.8,  \cite{Huang00})} 
The $X$ be a random variable in the space $D^{1,1}$ such that $X=0$ a.s. on some set $A\in {\cal F}$. Then $DX=0$ a.s. on $A$.
\end{lemma}

\begin{lemma}
\label{locdelta}
  \textnormal{(Proposition 3.9,  \cite{Huang00})} 
Let $u\in {\cal L}^{1,2}$ and $A\in {\cal F}$, such that $u=0$ a.a. on $[0,T]\times A$. Then $ \delta u=0$ a.s. on $A$.
\end{lemma}

 When $u\in L_a^2(\Omega\times [0,T])$, its It\^o integral $\int_0^tu_s\mathrm{d}W_s$ is a continuous process (see \cite{Karatzas91}). Similar result in the Skorohod integral can be given by the following lemma.

 \begin{lemma}\label{continuous}
 \textnormal{(Proposition 3.2.2,  \cite{Nualart06})} 
Let $u\in {\cal L}^{1,p}$, $p>2$. Then the Skorohod integral $\int_0^tu_s\mathrm{d}W_s$ is continuous (in the sense of modification).
 \end{lemma}
  
Like semi-martingales (or It\^{o} processes) in It\^{o} theory, here we give a similar characterization in the Skorohod integral.

 \begin{lemma}\label{sko2}
 \textnormal{(Proposition 3.1.1,  \cite{Nualart06})} 
Consider a process of the form $X_t=X_0+\int_0^tu_s\mathrm{d}W_s+\int_0^tv_s\mathrm{d}s$, where $X_0\in D^{1,2}(\Omega)$, $u\in {\cal L}^{2,2}$, and $v\in {\cal L}^{1,2}$. Then we have $X \in {\cal L}^{1,2,2}$, and $(D^-X)_s=D_sX_0+\int_0^s D_su_r\mathrm{d}W_r+\int_0^sD_sv_r\mathrm{d}r$ (resp. $(D^+X)_s=u_s+D_sX_0+\int_0^sD_su_r\mathrm{d}W_r+\int_0^s D_sv_r\mathrm{d}r$).
 \end{lemma}

We will give an approximation property of $u\in {\cal L}^{1,2}$, which will be used in theory of the forward integral (Proposition \ref{multipf1}). Before that, an important inequality in Harmonic analysis is required.

  \begin{lemma}\label{hlconv_1}
   \textnormal{(Hardy-Littlewood,  Theorem 2.5,  \cite{Duoandikoetxea01})} 
If $f $ is locally integrable on $\mathbb {R}^n$, we define its Hardy-Littlewood maximal function $Mf$ by\\ $Mf(x):=\sup_{r>0}\frac{1}{{\mathrm m}(B_x(r))}\int_{B_x(r)}|f(y)|\mathrm{d}y,\quad x\in {\mathbb R}^n$, where $B_x(r)$ is a Euclidean ball of radius $r$ centered at $x$, and ${\mathrm m}$ is the Lebesgue measure on ${\mathbb R}^n$. Then for all $p\in (1,\infty]$, we have
    \begin{equation*} 
\begin{aligned}
\| Mf\|_{L^p(\mathbb {R}^n)}\le C_p\| f\|_{L^p(\mathbb {R}^n)},
\end{aligned}
\end{equation*}
where constant $C_p$ depends only on $p$. Moreover, for a locally integrable function $f$, we have
    \begin{equation*} 
\begin{aligned}
\lim_{r\rightarrow 0^+}\frac{1}{ {\mathrm m}(B_x(r))}\int_{B_r(x)}f(y)\mathrm{d}y=f(x)
\end{aligned}
\end{equation*}
for a.e. $x$. The above conclusions still hold when we take some types of cubes containing $x$ instead of $B_x(r)$.
 \end{lemma}

 \begin{lemma}\label{conv_1}
If $u\in L^2(\Omega\times [0,T])$, then $\lim_{\varepsilon\rightarrow 0^+}\int_{(r-\varepsilon)\vee 0}^r\frac{u_s}{\varepsilon}\mathrm{d}s=u_r$ in $L^2(\Omega\times [0,T])$. Furthermore, the convergence also holds in ${\cal L}^{1,2}$ whenever $u\in {\cal L}^{1,2}$.
 \end{lemma}
  \begin{proof}
  For convenience, we only prove the convergence in $L^2(\Omega\times [0,T])$. By Lemma \ref{hlconv_1}, we have $\int _{(r-\varepsilon)\vee 0}^{r}\frac{u_s}{\varepsilon} \mathrm{d}s\rightarrow u_r  $ for a.a. $(\omega,r)\in\Omega\times[0,T]$. Since 
 \begin{equation*}  
\begin{aligned}
\left| \int _{(r-\varepsilon)\vee 0}^{r}\frac{u_s }{\varepsilon} \mathrm{d}s\right|\le 
 \sup_{0\le \varepsilon\le T}\int_{(r-\varepsilon)\vee 0}^{r}\frac{\left|u_s\right|}{\varepsilon}\mathrm{d}s,
  \end{aligned}
\end{equation*}
and by Lemma \ref{hlconv_1}, 
 \begin{equation*}  
\begin{aligned}
\mathbb{E}\left\| \sup_{0\le \varepsilon\le T}\int_{(r-\varepsilon)\vee 0}^{r}\frac{\left|u_s \right|}{\varepsilon}\mathrm{d}s \right\|^2_{L^2([0,T])}\le C\mathbb{E}\left\| u_r   \right\|^2_{L^2([0,T])}<\infty,
  \end{aligned}
\end{equation*}
we have by the dominated convergence theorem, $\lim_{\varepsilon\rightarrow 0^+}\int_{(r-\varepsilon)\vee 0}^r\frac{u_s}{\varepsilon}\mathrm{d}s=u_r$ in $L^2(\Omega\times [0,T])$.
   \end{proof}
 
 The following two lemmas give the Multiplication formulae of $D_t$ and $D^-$, which will be used in theory of the forward integral (Proposition \ref{multipf2}).
 
  \begin{lemma}\label{utimessigma1}
Let $u,\sigma\in {\cal L}^{1,2}$. Also assume that $\sigma$ and $D_s\sigma_t$ are bounded. Then $u\sigma\in {\cal L}^{1,2}$ and $D_s(u_t\sigma_t)=\sigma_tD_su_t+u_tD_s\sigma_t$.
 \end{lemma}
 \begin{proof}
 For each $t\in [0,T]$, $\mathbb{E}|u_t\sigma_t|^2\le C\mathbb {E}|u_t|^2$, $\mathbb{E}\int_0^T|\sigma_tD_s u_t|^2\mathrm{d}s\le C\mathbb{E}\int_0^T| D_s u_t|^2\mathrm{d}s$, and $\mathbb{E}\int_0^T|u_tD_s \sigma_t|^2\mathrm{d}s\le C \mathbb{E}\int_0^T|u_t |^2\mathrm{d}s$. Thus $u_t\sigma_t\in D^{1,2}(\Omega)$ and $D_s(u_t\sigma_t)=\sigma_tD_su_t+u_tD_s\sigma_t$ by Lemma \ref{chainrule}. Since all the above norms are controlled, we deduce that $u\sigma\in {\cal L}^{1,2}$. 
 \end{proof}
 
  \begin{lemma}\label{utimessigma2}
Let $u,\sigma$ be processes in $ {\cal L}^{1,2,2-}$ which are $L^2$-bounded and left-continuous in the norm $L^2(\Omega)$. Also assume that $\sigma$ and $D_s\sigma_t$ are bounded. Then $u\sigma\in {\cal L}^{1,2,1-}$ and $(D^-(u\sigma))_s=\sigma_s(D^-u)_s+u_s(D^-\sigma)_s$.
 \end{lemma}
 \begin{proof}
First we have $u\sigma\in {\cal L}^{1,2}$ by Lemma \ref{utimessigma1}. By the definition of the operator $D^-$, it suffices to estimate $\mathbb {E}|\sigma_tD_su_t-\sigma_s(D^-u)_s|$ and $\mathbb{E}|u_tD_s\sigma_t-u_s(D^-\sigma)_s|$. For the first term, we have 
    \begin{equation*} 
\begin{aligned}
&\int_0^T\sup_{(s-\varepsilon)\vee 0\le t<s}\mathbb {E}|\sigma_tD_su_t-\sigma_s(D^-u)_s|\mathrm{d}s \\
\le &\int_0^T\sup_{(s-\varepsilon)\vee 0\le t<s}\mathbb {E}|\sigma_tD_su_t-\sigma_t(D^-u)_s|\mathrm{d}s+ \int_0^T\sup_{(s-\varepsilon)\vee 0\le t<s}\mathbb {E}|\sigma_t(D^-u)_s-\sigma_s(D^-u)_s|\mathrm{d}s  \\
\le & \left(\int_0^T\sup_{(s-\varepsilon)\vee 0\le t<s}\mathbb {E}|\sigma_t   |^2\mathrm{d}s\right)^{\frac{1}{2}}\left(\int_0^T\sup_{(s-\varepsilon)\vee 0\le t<s}\mathbb {E}|D_su_t -(D^-u)_s|^2\mathrm{d}s\right)^{\frac{1}{2}}\\
&+ \left(\int_0^T\sup_{(s-\varepsilon)\vee 0\le t<s}\mathbb {E}|\sigma_t -\sigma_s  |^2\mathrm{d}s\right)^{\frac{1}{2}}\left(\int_0^T \mathbb {E}|(D^-u)_s|^2\mathrm{d}s\right)^{\frac{1}{2}},
\end{aligned}
\end{equation*}
which converges to $0$ when $\varepsilon\rightarrow 0^+$ due to the fact that $u\in {\cal L}^{1,2,2-}$ and $\sigma$ is $L^2$-bounded and left-continuous. The convergence of the second term is in a similar way.
  \end{proof}
  
  \begin{remark}\label{remsigma}
  Notice that if an ${\cal F}_t$-adapted process $v\in {\cal L}^{1,2}$, then $v\in {\cal L}^{1,2,2-}$ with $(D^-v)_s=0$. Thus the condition `$\sigma \in {\cal L}^{1,2,2-}$ and $u$ is left-continuous in the norm $L^2(\Omega)$' in Lemma \ref{utimessigma2} can be replaced by `$\sigma$ is ${\cal F}_t$-adapted and belongs to ${\cal L}^{1,2}$'.
  \end{remark}

 \section{Proofs of main results}\label{Appp}
 \setcounter{equation}{0} 
\setcounter{theorem}{0}

\renewcommand\theequation{B.\arabic{equation}}
\renewcommand\thetheorem{B.\arabic{theorem}}
\renewcommand\theremark{B.\arabic{remark}}

\begin{proof}[Proof of Proposition \ref{multipf1}]
   By the multiplication formula (Lemma \ref{multidelta}) and the Fubini theorem (Exercise 3.2.7 in \cite{Nualart06}), we have 
    \begin{equation} \label{rel1}
\begin{aligned}
& \lim_{\varepsilon\rightarrow 0^+}{\varepsilon}^{-1}\int_0^t u_s(W_{(s+\varepsilon)\wedge T}-W_s)\mathrm{d}s 
\\=& \lim_{\varepsilon\rightarrow 0^+} \int_0^t  u_s\int_s^{(s+\varepsilon)\wedge T} \frac{1}{\varepsilon}\mathrm{d}W_r\mathrm{d}s \\
 =& \lim_{\varepsilon\rightarrow 0^+} \left\{\int_0^t  \int_s^{(s+\varepsilon)\wedge T} \frac{u_s}{\varepsilon}\mathrm{d}W_r\mathrm{d}s+\int_0^t  \int_s^{(s+\varepsilon)\wedge T} \frac{D_ru_s}{\varepsilon}\mathrm{d}r\mathrm{d}s\right\}\\
  =&\lim_{\varepsilon\rightarrow 0^+}\left\{ \int_0^T  \int_{(r-\varepsilon)\vee 0}^{r} \frac{u_s1_{[0,t]}(s)}{\varepsilon}\mathrm{d}s\mathrm{d}W_r +\int_0^T   \int_{(r-\varepsilon)\vee 0}^{r}  \frac{D_ru_s1_{[0,t]}(s)}{\varepsilon}\mathrm{d}s\mathrm{d}r \right\}.
\end{aligned}
\end{equation}

For the first term on the right side of (\ref{rel1}), we have $\lim_{\varepsilon\rightarrow 0^+}\int_{(r-\varepsilon)\vee 0}^{r} \frac{u_s1_{[0,t]}(s)}{\varepsilon}\mathrm{d}s =u_r1_{[0,t]}(r)$ in ${\cal L}^{1,2}$ by Lemma \ref{conv_1}. Then the convergence of the first term in $L^2(\Omega)$ follows from the boundedness of $ \delta$ (Lemma \ref{boundeddelta}).

For the second term, we have
    \begin{equation}\begin{aligned}
     &{\mathbb E}\left|\int_0^T   \int_{(r-\varepsilon)\vee 0}^{r}  \frac{D_ru_s1_{[0,t]}(s)}{\varepsilon}\mathrm{d}s\mathrm{d}r -\int_0^t(D^-u)_r \mathrm{d}r\right| \\ \le & {\mathbb E} \int_0^T   \int_{r-\varepsilon}^{r}  \frac{|D_ru_s1_{[(r-\varepsilon)\vee 0,r]}(s)-(D^-u)_r|}{\varepsilon}1_{[0,t]}(r)\mathrm{d}s\mathrm{d}r\\
     &+  {\mathbb E} \int_0^T \int_{(r-\varepsilon)\vee 0}^{r}\frac{|D_ru_s||1_{[0,t]}(r)-1_{[0,t]}(s)|}{\varepsilon } \mathrm{d}s\mathrm{d}r \\ \le &\int_0^t\sup_{ r-\varepsilon\le s<r}{\mathbb E}|D_ru_s1_{[(r-\varepsilon)\vee 0,r]}(s)-(D^-u)_r|\mathrm{d}r
     \\&+  \int_0^T\sup_{(r-\varepsilon)\vee0\le s<r}\mathbb{E}|D_ru_s|\int_{(r-\varepsilon)\vee0}^{r}\frac{|1_{[0,t]}(r)-1_{[0,t]}(s)|}{\varepsilon} \mathrm{d}s\mathrm{d}r  \\ \le &\int_0^t\sup_{ (r-\varepsilon)\vee 0\le s<r}{\mathbb E}|D_ru_s-(D^-u)_r|\mathrm{d}r+\int_0^{\varepsilon} {\mathbb E}|(D^-u)_r|\mathrm{d}r+  \int_t^{(t+\varepsilon)\wedge T} \sup_{(r-\varepsilon)\vee0\le s<r}\mathbb{E}|D_ru_s|   \mathrm{d}r.  
     \end{aligned}\end{equation}
  The above convergence in $L^1(\Omega)$ follows from the definition of ${\cal L}^{1,2,1-}$.
\end{proof}

   \begin{proof}[Proof of Proposition \ref{multipf2}]
It is an immediate consequence of Proposition \ref{multipf1}, Lemma \ref{utimessigma2} and Remark \ref{remsigma}.
   \end{proof}

   \begin{proof}[Proof of Theorem \ref{itofor}]
  By means of localization we can assume that the processes $f(X_t)$, $f'(X_t)$, $f''(X_t)$ and $\int_0^Tu_t^2\mathrm{d}t$ are uniformly bounded, $X_0\in D^{1,2}(\Omega)$, $u\in{\cal L}^{f}$ and $v\in{\cal L}^{1,2}$ (see \cite{Nualart06, Huang00}). By Proposition \ref{multipf1}, the process $X_t$ has the following decomposition
    \begin{equation*}  
\begin{aligned}
X_t=X_0+\int_0^tu_s\mathrm{d}W_s+\int_0^tv_s\mathrm{d}s+\int_0^t(D^-u)_s\mathrm{d}s.
     \end{aligned}
\end{equation*}
This process verifies the conditions of Theorem \ref{itos}. Consequently, we can apply It\^{o} formula to $X$ and obtain
    \begin{equation}
  \label{itoforequ1}  
\begin{aligned}
f(X_t)=&f(X_0)+\int_0^tf'(X_s)u_s\mathrm{d}W_s+\int_0^tf'(X_s)v_s\mathrm{d}s+\int_0^tf'(X_s)(D^-u)_s\mathrm{d}s\\
&+\frac{1}{2}\int_0^tf''(X_s)u_s^2\mathrm{d}s+\int_0^tf''(X_s)(D^-X)_su_s\mathrm{d}s.
     \end{aligned}
\end{equation}
The process $f'(X_t)$ belongs to ${\cal L}^{1,2,1-}$. In fact, notice first that as in the proof of Lemma \ref{utimessigma1} the boundedness of $f'(X_t)$, $f''(X_t)$ and $\int_0^Tu_t^2\mathrm{d}t$ and the fact that $u_t\in {\cal L}^{2,2}\cap {\cal L}^{1,4}$, $v, D^-u\in {\cal L}^{1,2}$ and $X_0\in D^{1,2}(\Omega)$ imply that this process belongs to ${\cal L}^{1,2}$ and 
    \begin{equation*}  
\begin{aligned}
D_s(f'(X_t)u_t)=f'(X_t)D_su_t+f''(X_t)D_sX_tu_t.
     \end{aligned}
\end{equation*}
On the other hand, as in the proof of Lemma \ref{utimessigma2}, using that $u\in {\cal L}^{1,2,1-}$, $u_t$ is left-continuous in $L^2(\Omega)$, $u$ is $L^2$-bounded, $X_t$ is continuous and $X\in {\cal L}^{1,2,2}$ (by Theorem \ref{itos}), we deduce that $f'(X_t)u_t\in {\cal L}^{1,2,1-}$ and that
    \begin{equation*}  
\begin{aligned}
(D^-(f'(X)u))_s=f'(X_s)(D^-u)_s+f''(X_s)(D^-X)_su_s.
     \end{aligned}
\end{equation*}
Substituting it into (\ref{itoforequ1}) and using Proposition \ref{multipf1} we obtain
  \begin{equation*}
\begin{aligned}
f(X_t)&=f(X_0)+\int_0^tf'(X_s)u_s\mathrm{d}W_s+\int_0^t (D^-(f'(X)u))_s\mathrm{d}s+\int_0^tf'(X_s)v_s\mathrm{d}s
+\frac{1}{2}\int_0^tf''(X_s)u_s^2\mathrm{d}s\\
&=f(X_0)+\int_0^tf'(X_s)u_s\mathrm{d}^-W_s +\int_0^tf'(X_s)v_s\mathrm{d}s
+\frac{1}{2}\int_0^tf''(X_s)u_s^2\mathrm{d}s.
     \end{aligned}
\end{equation*}
   \end{proof}

\end{appendices}

\end{CJK}
\end{document}